\documentclass[a4paper, oneside]{amsart}
\usepackage{fullpage}
\usepackage{mathrsfs}
\usepackage{amsfonts}
\usepackage{xy}
\usepackage{cite}
\usepackage{amssymb}
\usepackage{color, latexsym}
\usepackage{graphicx}
\usepackage{amsbsy,amscd}
\usepackage{fancyhdr}
\usepackage[colorlinks=true,
           citecolor=blue,
           linkcolor =red]{hyperref}
\xyoption{all}
\allowdisplaybreaks
\makeatletter
\@namedef{subjclassname@2020}{\textup{2020} Mathematics Subject Classification}
\makeatother
\numberwithin{equation}{section}
\newtheorem{theorem}[equation]{Theorem}
\newtheorem{lemma}[equation]{Lemma}
\newtheorem{proposition}[equation]{Proposition}
\newtheorem{corollary}[equation]{Corollary}
\theoremstyle{definition}
\newtheorem{definition}[equation]{Definition}
\theoremstyle{remark}
\newtheorem{remark}[equation]{Remark}

\begin{document}
\title[Markov traces]{\fontsize{9}{\baselineskip}\selectfont Markov traces on degenerate cyclotomic Hecke algebras}
\author{Deke Zhao}
\address{\bigskip\hfil\begin{tabular}{l@{}}
           Department of Mathematics\\
            Beijing Normal University at Zhuhai, Zhuhai, 519087\\
             China\\
             E-mail: \it deke@amss.ac.cn \hfill
          \end{tabular}}
\thanks{Supported by Guangdong Basic and Applied Basic Research Foundation
 (Grant No. 2023A1515010251), the National Natural Science Foundation of China (Grant No.  11871107), and the Open Project Program of Key Laboratory of Mathematics and Complex System of Beijing Normal University (Grant No. K202402)}
\dedicatory{Dedicated to Professor Yingbo Zhang on the occasion of her 80th birthday}
\vspace*{-3mm}
\begin{abstract}Let $H_n(\boldsymbol{u})$ be the degenerate cyclotomic Hecke algebra with parameter $\boldsymbol{u}=\{u_1, \ldots, u_m\}$. The paper aims to define and construct the (non-)normalized Markov traces on the sequence $\{H_n(\boldsymbol{u})\}_{n=1}^{\infty}$. This allows us to provide the Brou\'{e}--Malle--Michel symmetrizing trace on $H_n(\boldsymbol{u})$ and show that the Brundan--Kleshchev trace is a specialization of the non-normalized Markov traces on $H_n(\boldsymbol{u})$.
\end{abstract}
\maketitle
\section{Introduction}
Let $m,n$ be positive integers and let $W_{m,n}$ be the complex reflection group of type $G(m,1,n)$ in Shephard--Todd's classification \cite{ST}.
It is well-known that $W_{m,n}$ is isomorphic to the wreath product $(\mathbb{Z}/m\mathbb{Z})^{n}\rtimes \mathfrak{S}_{n}$, where  $\mathfrak{S}_{n}$ is the symmetric group of degree $n$ (generated by the simple transpositions $s_i=(i,i+1)$ for $i=1,\ldots,n-1$).

Let $\boldsymbol{u}=\{u_1, \ldots, u_m\}$ be a set of indeterminates and let $\mathbb{C}(\boldsymbol{u})$ be the fraction field of $\mathbb{C}[\boldsymbol{u}]$. The (generic) degenerate cyclotomic Hecke algebra $H_n(\boldsymbol{u})$ is the unitary  associative algebra over $\mathbb{C}(\boldsymbol{u})$ generated by $t, s_1, \ldots, s_{n-1}$ and subjected to relations:
\begin{eqnarray*}
&&(t-u_1)\dots(t-u_m)=0,\\
&&t(s_1ts_1+s_1)=(s_1ts_1+s_1)t,\\
&&ts_i=s_it\text{ for }1\leq i\leq n-1,\\
&&s_i^2=1 \text{ for } 1\le i\leq n-1,\\
&&s_is_{i+1}s_i=s_{i+1}s_is_{i+1} \text{ for  }1\le i\leq n-2,\\
&&s_is_j=s_js_i \text{ for }|i-j|\geq 2.
\end{eqnarray*}
It is known that $H_{n}(\boldsymbol{u})$ is free as a $\mathbb{C}(\boldsymbol{u})$-module of rank $|W_{m,n}|=m^nn!$ and is a deformation of $W_{m,n}$ (see \cite[\S7]{K}).

Let us remark that the degenerate cyclotomic Hecke algebras are degenerate versions of the Ariki--Koike algebras or the cyclotomic Hecke algebra of type $G(m,1,n)$, which originates from the works of Drinfel'd \cite{Drinfeld} and Cherednik \cite{Cherednik} and was mentioned explicitly in Grojnowski's preprint \cite{Grojnowski}, Kleshchev's book \cite{K} would be a good reference. These two classes of algebras are closely related, as evidenced by Brundan and Kleshchev's seminal work \cite{BK-Block},  which shows that both algebras are isomorphic to a cyclotomic  Khovanov--Lauda--Rouquier (KLR)  algebra where the weight is determined by the  ``cyclotomic parameters".  Statements that are regarded as theorems in the setting of the cyclotomic Hecke algebras are often adopted as statements in the setting of the degenerate cyclotomic Hecke algebras, and vice versa (see e.g. \cite[\S6]{AMR}, \cite{BK-Block, BK-Math.Z., Z1,Z2} etc.).

In \cite{Jones}, Jones  constructed the Jones polynomial for knots
in $S^3$ by using Ocneanu's Markov trace on the Iwahori--Hecke algebras of type
$A$ and posed the question about similar constructions on other Hecke algebras as well
as in other 3-manifolds. In \cite{La94} Lambropoulou constructed the first Markov trace on the Iwahori--Hecke algebras of type $B$ and related it to the knot theory of the solid torus.  Then Geck and Lambropoulou  gave a full classification of all Markov traces on these algebras in \cite{GL,Geck}. Furthermore Lambropoulou \cite{L} constructed the Markov traces on the cyclotomic Hecke algebras, Juyumaya \cite{Juyumaya} proved the existence
of Markov trace on the Yokonuma--Hecke algebras and so on.  Moreover, as we remarked in \cite[Remarks~5.7(ii)]{Z2}, it is interesting to investigate the ``Markov traces" on (cyclotomic) KLR algebras, which will be helpful to understand the (Markov) traces on (degenerate) cyclotomic Hecke algebras.

Inspired by the aforementioned works, it is natural and interesting to study Markov traces on  $H_n(\boldsymbol{u})$. The aim of the paper is to define and construct the (non-)normalized Markov traces on
$H_n(\boldsymbol{u})$ along the line of Lambropoulou's argument in \cite{L}, and investigate their specializations. A point should be noted that the most natural way to define and study the Markov traces on $H_n(\boldsymbol{u})$ is to combine the Brundan--Kleshchev isomorphism and Lambropoulou's  work \cite{L}. Unfortunately the isomorphism is too complicated to attack the problem. Roughly speaking, Brundan and Kleshchev constructed the Khovanov--Lauda generators of the cylotomic Hecke algebra and the cyclotomic Hecke algebra respectively and proved that they satisfy the defining relations of the cyclotomic Khovanov--Lauda algebra. While the Khovanov--Lauda generators are too complicated to determine the value of Markov traces on them, which makes it is very hard to study the Markov traces on $H_n(\boldsymbol{u})$.

We now explain our results in detail. Clearly we have the chains of subgroups
\begin{equation*}
 W_{m,1}\subset W_{m,2}\subset\cdots\subset W_{m,n}
\end{equation*}
 and the corresponding chain of subalgebras
  \begin{equation*}
  H_1(\boldsymbol{u})\subset H_2(\boldsymbol{u})\subset\cdots\subset H_n(\boldsymbol{u}),
\end{equation*}
 where $W_{m,i}$ (resp.~$H_i(\boldsymbol{u})$) is generated by $s_0, s_1, \ldots, s_{i-1}$ (resp.~$t$, $s_1$, $\ldots$, $s_{i-1}$) for $1\leq i\leq n$. (We also set $W_{m,0}=\{1\}$ and $H_0(\boldsymbol{u})=\mathbb{C}(\boldsymbol{u})$.)

\begin{definition}\label{Def:Markov-traces}
  Let $z$ be a new variable. A \textit{Markov trace} on the tower of algebras $\{H_n(\boldsymbol{u})\}_{n=1}^{\infty}$ is defined to a collection of $\mathbb{C}(\boldsymbol{u})$--linear maps
\begin{equation*}
  \mathrm{Mtr}_n: H_n(\boldsymbol{u})\rightarrow \mathbb{C}(\boldsymbol{u},z), n\geq 1
\end{equation*}
such that
\begin{itemize}
\item $\mathrm{Mtr}_n(\alpha\beta)=\mathrm{Mtr}(\beta\alpha)$ for all $\alpha,\beta\in H_n(\boldsymbol{u})$ and all $n\geq 1$;
\item $\mathrm{Mtr}_{n+1}(\beta)=\mathrm{Mtr}_n(\beta)$  for all $\beta\in H_{n}(\boldsymbol{u})\subset H_{n+1}(\boldsymbol{u}$ and all $n\geq 1$;
\item $\mathrm{Mtr}_{n+1}(\beta s_n)=z\cdot\mathrm{Mtr}_n(\beta)$ for all $\beta\in H_n(\boldsymbol{u})$ and all $n\geq 1$.
\end{itemize}
\end{definition}

 Now we can state the first main result of this paper:
\begin{theorem}\label{Them:Markov}
  Given $z, y_1, \ldots, y_{m-1}\in \mathbb{C}(\boldsymbol{u})$, there is  a  unique $\mathbb{C}(\boldsymbol{u})$--linear function
   \begin{equation*}
     \mathrm{tr}: H_{\infty}(\boldsymbol{u})=\bigcup_{n=1}^{\infty}H_n(\boldsymbol{u})\rightarrow\mathbb{C}(\boldsymbol{u})
\end{equation*}
satisfying
\begin{enumerate}
\item[(m1)] $ \mathrm{tr}(1)=1$ (normalized condition);
\item[(m2)] $\mathrm{tr}(\alpha\beta)=\mathrm{tr}(\beta\alpha)$ for all $\alpha,\beta\in H_{\infty}(\boldsymbol{u})$;
\item[(m3)] $\mathrm{tr}(\alpha s_n)=z\,\mathrm{tr}(\alpha)$  for all $\alpha\in H_{n}(\boldsymbol{u})$ and for all $n\geq 1$;
\item[(m4)] $\mathrm{tr}(\alpha t_{n+1}^k))=y_{k}\mathrm{tr}(\alpha)$ for all $\alpha\in H_n(\boldsymbol{u})$,  $1\leq k\leq m-1$, and for all $n\geq 1$, where \begin{align*}t_{n+1}=s_{n}\cdots s_1ts_1\cdots s_{n} \text{ with }t_1=t.\end{align*}
\end{enumerate}
\end{theorem}
Notably, since all generators $s_i$ and all $t_i^k$ ($k=1, \ldots,m-1$) for $i=1,2,\ldots$, are respectively conjugate in $H_{\infty}(\boldsymbol{u})$, any trace function must assign identical values to these elements. This explains the parameters $z$ and $y_k$'s independence from $n$ in Theorem~\ref{Them:Markov}~(m3, m4). We will refer to the uniquely defined $\mathbb{C}(\boldsymbol{u})$--linear function in Theorem~\ref{Them:Markov} as a \emph{normalized Markov trace} of $H_{\infty}(\boldsymbol{u})$ with parameters $z, y_1, \ldots, y_{m-1}$. Let us remark that the specialization of the normalized Markov trace enables us to define the Brou\'{e}--Malle--Michel symmetrizing trace on $H_n(\boldsymbol{u})$ (see Corollary~\ref{Cor:t-trace}), which is similar to the canonical symmetrizing trace on the cyclotomic Hecke
algebras introduced in \cite[\S2B]{BM}.

Recall that the \emph{Jucys--Murphy elements} of $H_n(\boldsymbol{u})$ are defined inductively as
\begin{eqnarray}\label{Equ:JM-H}
 &&J_1=t\text{ and }J_{i+1}\!:=s_iJ_{i}s_i+s_i, \quad i=1, \cdots, n-1.
  \end{eqnarray}
\cite[Theorem~7.5.6]{K} shows that the set
\begin{equation}\label{Equ:Standard-basis}
 \mathfrak{B}_n=\left\{J_1^{a_1}J_2^{a_2}\cdots J_{n}^{a_{n}}w\,|\,0\leq a_1, \ldots, a_n\leq m-1, w\in \mathfrak{S}_n\right\}
\end{equation}
forms a $\mathbb{C}(\boldsymbol{u})$-basis of $H_n(\boldsymbol{u})$, which is referred to as the \emph{standard basis} of $H_n(\boldsymbol{u})$.

In \cite[Appendix]{BK-Schur-Weyl}, Brundan and Kleshchev defined the \emph{Brundan--Kleshchev trace} on  $H_n(\boldsymbol{u})$:
\begin{equation*}
\tau_{\mathrm{BK}}(J_1^{a_1}\cdots J_n^{a_n} w)
        :=\left\{\vspace{1\jot}\begin{array}{ll} 1,&\text{ if }a_1=\cdots=a_n=m-1 \text{ and } w=1;\\
                0,&\text{ otherwise},
                \end{array}\right.
\end{equation*} where the suffix BK stands for Brundan--Kleshchev. Obviously, the normalized condition $\mathrm{tr}(1)=1$ shows that $\tau_{\mathrm{BK}}$ is not a specialization of the normalized Markov trace on $H_n(\boldsymbol{u})$. Furthermore, $\tau_{\mathrm{BK}}$ is not a specialization of any Markov trace  on $H_n(\boldsymbol{u})$ determined by Theorem~\ref{Them:Markov} (m2)--(m4) without normalized condition. This motivates us to introduce and study the non-normalized Markov traces on $H_n(\boldsymbol{u})$.

 \begin{theorem}\label{Them:Markov-BK}Given $z, y_1, \ldots, y_{m-1}\in \mathbb{C}(\boldsymbol{u})$, there is a unique $\mathbb{C}(\boldsymbol{u})$--linear function
   \begin{equation*}
     \mathrm{Tr}: H_{\infty}(\boldsymbol{u})=\bigcup_{n=1}^{\infty}H_n(\boldsymbol{u})\rightarrow \mathbb{C}(\boldsymbol{u})
\end{equation*}
satisfying
\begin{enumerate}
\item[(M1)] $ \mathrm{Tr}(1)=0$ (non-normalized condition);
\item[(M2)] $\mathrm{Tr}(xy)=\mathrm{Tr}(yx)$ for all $x,y\in H_{\infty}(\boldsymbol{u})$;
\item[(M3)] $\mathrm{Tr}(xs_n)=z\cdot\mathrm{Tr}(x)$ and $\mathrm{Tr}(xJ_{n+1}^k)=\mathrm{Tr}(J_{n+1}^k)\mathrm{Tr}(x)$  for all $x\in H_{n}(\boldsymbol{u})$, for all $k=1,\ldots,m-1$, and for all $n\geq 1$;
\item[(M4)] $\mathrm{Tr}(xs_nJ_{n}^ks_n)=\mathrm{Tr}(J_{n}^k)\mathrm{Tr}(x)$ for all $x\in H_n(\boldsymbol{u})$,  $1\leq k\leq m-1$, and for all $n\geq 1$;
  \item[(M5)] $\mathrm{Tr}(J_1^k)=y_k$ for $k=1, \ldots, m-1$.
\end{enumerate}
  \end{theorem}
We will refer to the uniquely defined $\mathbb{C}(\boldsymbol{u})$--linear functions in Theorem~\ref{Them:Markov-BK} as the \emph{non-normalized Markov trace} of $H_{\infty}(\boldsymbol{u})$ with parameters $z, y_1, \ldots, y_{m-1}$. As an application, we show  that the Brundan--Kleshchev trace is a specialization of the non-normalized Markov trace on $H_n(\boldsymbol{u})$ (see Corollary~\ref{Cor:BK-specialization}).

We give two remarks related to the paper.

\begin{remark}Note that the set of all mixed braids on $n$ standard strings forms the Artin Braid group of type $B_n$, which we  denoted by $B_{1,n}$:
\begin{equation*}
  B_{1,n}=\left\langle \tau,\sigma_1,\ldots, \sigma_{n-1}\left| \begin{array}{l}
    \sigma_{i}\sigma_{i+1}\sigma_{i}=\sigma_{i+1}\sigma_{i}\sigma_{i+1}\text{ for all }i\\
    \sigma_i\sigma_j=\sigma_j\sigma_i \text{ for }|i-j|\geq2\\
    \tau\sigma_i=\sigma_i\tau \text{ for } 2\leq i\leq n-1\\
    \tau\sigma_1\tau\sigma_1=\sigma_1\tau\sigma_1\tau
  \end{array} \right\rangle.\right.
\end{equation*}
$B_{1,n}$ may be seen as the subgroup of the classical braid group $B_{n+1}$  on $n+1$
strands, the elements of which keep the first strand fixed, which is isomorphic to the group of  braids with $n$ string in a thickened cylinder.
Now let $q$ be a new variable. The affine Hecke algebra $H^{\mathrm{aff}}_n(q)$ is an associative algebra over $\mathbb{C}(\boldsymbol{u},q)$, the fractional field of $\mathbb{C}[\boldsymbol{u},q]$, which is generated by $T_1,\ldots, T_{n-1}, X$ subjected to relations
\begin{eqnarray*}&&(T_i-q)(T_i+1)=0 \text{ for }1\leq i\leq n-1,\\
  &&T_iT_{i+1}T_i=T_{i+1}T_{i}T_{i+1} \text{ for }1\leq i\leq n-2,\\
  && T_iT_j=T_jT_i \text{ if } |i-j|\geq 2,\\
  &&XT_1XT_1=T_1XT_1X,\\
  && XT_i=T_iX \text{ for }2\leq i\leq n-1.\end{eqnarray*}
Clearly, there is a natural surjective homomorphism $\pi: B_{1,n}\twoheadrightarrow H^{\mathrm{aff}}_n(q)$. Now applying the Yangian limit to the affine Hecke algebra $H^{\mathrm{aff}}_n(q)$ (see e.g. \cite[Remark~5.1]{Isaev-Kirillov}), we obtain the degenerate affine Hecke algebra $H^{\mathrm{deg}}_n$ over $\mathbb{C}(\boldsymbol{u})$, which is generated by $s_1,\ldots, s_{n-1},x$ subjected to the relations:
\begin{eqnarray*}
&&x(s_1xs_1+s_1)=(s_1xs_1+s_1)x,\\
&&xs_i=s_ix\text{ for }2\leq i\leq n-1,\\
&&s_i^2=1 \text{ for } 1\le i\leq n-1,\\
&&s_is_{i+1}s_i=s_{i+1}s_is_{i+1} \text{ for  }1\le i\leq n-2,\\
&&s_is_j=s_js_i \text{ for } |i-j|\geq 2.
\end{eqnarray*}
Thus $H_n(\boldsymbol{u})$ is a quotient of $H^{\mathrm{deg}}_n$, which implies that there is a natural surjective homomorphism
 \begin{equation*}
  B^{\mathrm{aff}}_n\stackrel{\pi}\twoheadrightarrow H^{\mathrm{aff}}_n(q)\xymatrix@C=2.0cm{
  \ar@{->}[r]^{\text{Young limit}}&}H^{\mathrm{deg}}_n\twoheadrightarrow H_n(\boldsymbol{u}).
 \end{equation*}
Thanks to the Brundan--Kleshchev isomorphism,  it is natura to expect that the Markov traces on $H_n(\boldsymbol{u})$ may help us to understand the `mixed' knots/links in the solid torus (cf.~\cite[\S5]{L}),  through the relations $s_i^2=1$ kill the braiding information.
\end{remark}

\begin{remark}Denote by  $\mathscr{P}_{m,n}$ the set of all $m$-tuples of partitions $\boldsymbol{\lambda}=(\lambda^{(1)};\ldots; \lambda^{(m)})$ such that $|\boldsymbol{\lambda}|=|\lambda^{(1)}|+\cdots+|\lambda^{(m)}|=n$, that is, $m$-multipartitions of $n$. Then, thanks to \cite[Theorem~6.11]{AMR}, $H_n(\boldsymbol{u})$ is split semisimple and $\mathscr{P}_{m,n}$ parameterizes the simple $H_n(\boldsymbol{u})$-modules.  We write  \begin{equation*}
\mathrm{Irr\,}(H_n(\boldsymbol{u}))=\{\chi_{\boldsymbol{\lambda}}|\boldsymbol{\lambda}\in\mathscr{P}_{m,n}\}
\end{equation*}
for the set of irreducible characters of $H_{n}(\boldsymbol{u})$.  Since $\mathrm{Irr}(H_n(\boldsymbol{u}))$ is a basis of the vector space of trace functions on $H_n(\boldsymbol{u})$, for any trace $\tau: H_n(\boldsymbol{u})\rightarrow \mathbb{C}(\boldsymbol{u})$, there are unique elements $\omega_{\boldsymbol{\alpha}}\in \mathbb{C}(\boldsymbol{u})$ satisfying
\begin{equation*}
  \tau=\sum_{\boldsymbol{\alpha}\in \mathscr{P}_{m,n}}\omega_{\boldsymbol{\alpha}}\chi_{\boldsymbol{\alpha}},
\end{equation*}
which are called the \emph{weights} of $\tau$. It would be interesting to find the explicit formulas for the weights of (non-)normalized Markov traces, which would enable us to give an alternative proof of the combinatorial formulas for the Schur elements of $H_n(\boldsymbol{u})$ (see \cite[Theorem~5.5]{Z2} or \cite[Theorems~3.4 and 4.2]{Z1}. It is natural to expect that the two approaches in \cite{Orellana-Ram,Rui} may be adapted to compute the weights, unfortunately I can not do this.
\end{remark}

This paper is organized as follows. Section~\ref{Sec:Inductive-basis} provides an inductive basis for the degenerate cyclotomic Hecke algebra via its standard basis. In Section~\ref{Sec:Markov-traces}, we construct the normalized Markov traces on the degenerate cyclotomic Hecke algebra via its inductive basis and prove Theorem~\ref{Them:Markov}. Section~\ref{Sec:BK-trace} aims to construct the non-normalized Markov traces on the degenerate cyclotomic Hecke algebra via its standard basis and prove Theorem~\ref{Them:Markov-BK}. The specializations of the (non-)normalized Markov traces are investigated in the last section.

\section{Inductive basis}\label{Sec:Inductive-basis}

In this section we present some facts which are used later and construct the inductive basis for $H_n(\boldsymbol{u})$ from its standard basis.

The \textit{Jucys--Murphy elements} of $\mathbb{C}(\boldsymbol{u})\mathfrak{S}_n$ are defined as the sum of transpositions:
\begin{eqnarray*}
  &&L_k=(1,k)+(2,k)+\cdots+(k\negmedspace-\negmedspace1,k),\quad k=2, \ldots, n,
\end{eqnarray*}
with $L_1=0$ (and recall that $s_i=(i,i+1)$ for $i=1,\ldots,n-1$). These elements,   studied by Jucys \cite{Jucys} and Murphy \cite{Murphy}, admit an inductive characterization:
\begin{eqnarray}\label{Equ:JM-S}
&&L_1=0 \text{ and } L_{k+1}=s_kL_{k}s_k+s_k, \quad k=1, 2,\ldots, n-1,
\end{eqnarray}
or equivalently,
\begin{eqnarray*}
L_{k+1}&=&\sum_{j=1}^{k}s_k\cdots s_j\cdots s_k, \quad k=1, 2,\ldots, n-1.
\end{eqnarray*}
They generate a maximal commutative subalgebra of $\mathbb{C}(\boldsymbol{u})\mathfrak{S}_n$  \cite[Eq.~(2.6)]{Murphy}.

For $1\leq i,j\leq n-1$, we write $s_{i,j}=s_i\cdots s_{j}$. Recall that $t_{i+1}=s_{i}t_{i}s_{i}$ with $t_1=t$ for $i\geq1$, that is, $t_{i+1}=s_{i,1}t_{1}s_{1,i}$. Eqs.~\eqref{Equ:JM-H} and \eqref{Equ:JM-S} show
\begin{eqnarray}\label{Equ:J-JM}
  &&J_k=t_k+L_k, \qquad k=1,2, \ldots, n.
\end{eqnarray}

The following fact is well-known, see for example \cite[2.3]{Z2}.
\begin{lemma}\label{Lem:si-Jj}Suppose that $1\le i<n$ and $1\le j, k\le n$. Then
\begin{enumerate}\setlength{\itemsep}{1\jot}
\item  $s_jJ_j-J_{j+1}s_j=-1$ and $s_{j-1}J_j-J_{j-1}s_{j-1}=1$.
 \item $s_iJ_j=J_js_i$ if $i\neq j-1, j$.
 \item $J_jJ_k=J_kJ_j$ if $1\le j, k\le n$.
 \item $s_j(J_jJ_{j+1})=(J_jJ_{j+1})s_j$ and $s_j(J_j+J_{j+1})=(J_{j}+J_{j+1})s_j$.
 \item if $a\in \mathbb{C}(\boldsymbol{u})$ and $i\neq j$ then $s_i$ commutes with $(J_1-a)(J_2-a)\cdots(J_j-a)$.
 \end{enumerate}
  \end{lemma}

The following fact will be useful.

\begin{lemma}\label{Lemm:s-t}For $1\leq a,b\leq n$ and  positive integers $k,\ell$, we have
  \begin{enumerate}
    \item[(i)] $s_at_b=t_bs_a$ if $a\neq b, b-1$.
    \item[(ii)] $s_at_a=t_{a+1}s_a$.
    \item[(iii)]$t_at_b-t_bt_a=[L_a,t_b]-[t_a,L_b]$, equivalently
    \begin{align*}
t_{a}t_b-t_bt_a=s_{b-1,1}s_{a-1,2}(s_1t-ts_1)s_{2,a-1}s_{1,b-1}.\end{align*}
\item[(iv)] Recursion identity:
\begin{align*}t^{\ell}s_1t^ks_1=s_1t^ks_1t^{\ell}+\displaystyle\sum_{i=1}^{\ell}
    (t^{\ell-i}s_1t^{k+i-1}-t^{k+i-1}s_1t^{\ell-i}).\end{align*}

\item[(v)] Braiding relation:  \begin{align*}t_{n}^{\ell}t_{n+1}^k=(s_{n-1}s_n)(s_{n-2}s_{n-1})\cdots(s_1s_2)
(t^{\ell}s_1t^ks_1)(s_2s_1)\cdots(s_{n-1}s_{n-2})(s_ns_{n-1}).\end{align*}

\item[(vi)]Conjugation relation:
\begin{align*}t_{n+1}^kt=tt_{n+1}^k+t^ks_{n,1}s_{2,n}-s_{n,1}s_{2,n}t^k.\end{align*}
  \end{enumerate}
\end{lemma}
\begin{proof} (i)--(ii) follow from direct computation. (iii) follows by using the commutativity of $L_a$ and $J_b$ with Eq.~\eqref{Equ:J-JM}.

We prove (iv) by induction on $\ell$. For $\ell=1$, by apply the equality \begin{align*}s_1ts_1ts_1t+s_1t=ts_1ts_1+ts_1\end{align*} iteratively, we obtain
\begin{align*}
ts_1t^{k}s_1&=s_1t^{k}s_1t+s_1t^{k}-t^{k}s_1
\end{align*}
for any positive integer $k$. Assume it is true for $\ell\geq 1$. For $\ell+1$, by induction hypothesis, we  have
\begin{align*}
t^{\ell+1}s_1t^{k}s_1&=ts_1t^{k}s_1t^{\ell}+\sum_{i=1}^{\ell}\left(t^{\ell+1-i}
s_1t^{k+i-1}-t^{k+i}s_1t^{\ell-i}\right)\\
&=s_1t^{k}s_1t^{\ell+1}+\sum_{i=1}^{\ell+1}\left(t^{\ell+1-i}
s_1t^{k+i-1}-t^{k+i-1}s_1t^{\ell+1-i}\right).
 \end{align*}
Thus (iv) holds for any positive integers $k,\ell$.

 (v). Apply induction on $n$ and $n=1$ is trivial. Assume it holds for $n\geq 1$. For $n+1$, \begin{align*}
 t_{n+1}^{\ell}t_{n+2}^{k}&=(s_ns_{n+1})t_{n}^{\ell}t_{n+1}^{k}(s_{n+1}s_{n})\\
  &=(s_ns_{n+1})(s_{n-1}s_n)\cdots(s_1s_2)(t^{\ell}s_1t^{k}s_1)
    (s_2s_1)\cdots(s_ns_{n-1})(s_{n+1}s_{n}).
 \end{align*}
 (vi). It is easy to see that
 \begin{align*}
   t_{n+1}^kt&=s_{n,2}(s_1t^ks_1t)s_{2,n}\\
   &=s_{n,2}(ts_1t^ks_1+t^ks_1-s_1t^k)s_{2,n}\\
   &=tt_{n+1}^k+t^ks_{n,1}s_{2,n}-s_{n,1}s_{2,n}t^k,
 \end{align*}
where the second equality follows by applying (iv).
\end{proof}

The following fact will be useful.
\begin{corollary}\label{Cor:t-n-n+1}For positive integers $n,k,\ell,a$, we have
\begin{align*}
t_{n}^{\ell}t_{n+a}^k=t_{n+a}^kt_{n}^{\ell}+\sum_{i=1}^{\ell}
s_{n+a-1}\cdots s_{n}
\left(t_n^{\ell-i}s_nt_n^{k+i-1}-t_n^{k+i-1}s_nt_n^{\ell-i}\right)
s_{n}\cdots s_{n+a-1}.
\end{align*}
\end{corollary}
\begin{proof}We argue by induction. For $a=1$, Lemma~\ref{Lemm:s-t}(v) gives
  \begin{align*}
 t_n^{\ell}t_{n+1}^k=(s_{n-1}s_n)\cdots(s_1s_2)(s_1t^ks_1t^{\ell})(s_2s_1)\cdots(s_ns_{n-1}).
  \end{align*}
Then the equality follows by using Lemma~\ref{Lemm:s-t}(iv, v).

Now assume it holds for $a\geq 1$.  Since
 \begin{align*}
  t_{n+a+1}^k=(s_{n+a}\cdots s_{n+1})t_{n+1}^k(s_{n+1}\cdots s_{n+a}),
 \end{align*}
  Lemma~\ref{Lemm:s-t}(i) shows
\begin{align*}
t_{n}^{\ell}t_{n+a+1}^k
&= (s_{n+a}\cdots s_{n+1})(t_{n}^{\ell}t_{n+1}^{k})(s_{n+1}\cdots s_{n+a}) \\
&= t_{n+a}^k t_{n}^{\ell} + \sum_{i=1}^{\ell} s_{n+a}\cdots s_{n+1}(t_n^{\ell-i}s_n t_n^{k+i-1} - t_n^{k+i-1}s_n t_n^{\ell-i})s_{n+1}\cdots s_{n+a}.
\end{align*}
This completes the proof.
\end{proof}

Now we are ready to describe the inductive basis for \(H_{\infty}(\boldsymbol{u})\), which are derived from the standard basis Eq.~(\ref{Equ:Standard-basis}). The following fact is an \(H_{n+1}(\boldsymbol{u})\)-analogue of \cite[(4.4)]{Jones} and \cite[Theorem~2]{L}.
\begin{theorem}\label{Them:Inductive-basis}
Every element of \(H_{n+1}(\boldsymbol{u})\) can be uniquely expressed as a $\mathbb{C}(\boldsymbol{u})$--linear combination of the following four types of words:
\begin{enumerate}
  \item[(I)] $h_{n}$;
  \item[(II)] $h_{n}s_n \cdots s_i$, $i = 1, \ldots, n$;
  \item[(III)] $h_{n}s_n \cdots s_i J_{i}^k$, $i = 1, \ldots, n$, $1 \leq k\leq m-1$;
  \item[(IV)] $h_{n}J_{n+1}^k$, $1 \leq k \leq m-1$;
\end{enumerate}
where \(h_{n} \in H_n(\boldsymbol{u})\).
\end{theorem}
\begin{proof}We prove this by induction on $n$.
Thanks to the standard basis theorem \cite[Theorem~7.5.6]{K} (see Eq.~\eqref{Equ:Standard-basis}), it is enough to show that \begin{align*}
x &= J_1^{k_1} \cdots J_{n+1}^{k_{n+1}} w \in \mathfrak{B}_{n+1}\setminus \mathfrak{B}_n
\end{align*}
can be expressed uniquely in terms of the four-type words, that is, we only need to consider the following cases:

\textbf{Case (a):} $x$ contains no $J_{n+1}$, but $w = w'(s_n s_{n-1} \cdots s_i)$ with $w'\in \mathfrak{S}_{n}$, i.e.,  \begin{align*}
  x = J_{i_1}^{k_1}J_{i_2}^{k_2} \cdots J_{i_r}^{k_r}w'(s_n s_{n-1} \cdots s_i),\end{align*} where $1 \leq i_1 < \cdots < i_r \leq n$ and $w'\in \mathfrak{S}_{n}$. Then
\begin{align*}x = (J_{i_1}^{k_1} \cdots J_{i_r}^{k_r}w')(s_n s_{n-1} \cdots s_i),\end{align*}
which is of type (II).

\textbf{Case (b):}  $x$ contains $J_{n+1}$ and \(x = J_{i_1}^{k_1}J_{i_2}^{k_2} \cdots J_{i_r}^{k_r}J_{n+1}^{k}w\) with \(1 \leq i_1 < \cdots < i_r \leq n\), $k\geq 1$ and \(w \in  \mathfrak{S}_{n}\). Using the commutation relation $J_{n+1}w=wJ_{n+1}$ (Lemma~\ref{Lem:si-Jj}(ii)), we rewrite
 \begin{align*}x = \left(J_{i_1}^{k_1} \cdots J_{i_r}^{k_r}w\right) J_{n+1}^{k},\end{align*}
  which is of type (IV).

\textbf{Case (c):} \(x = x'J_{n+1}^k w'(s_n s_{n-1} \cdots s_i)\) with \(x' \in H_n(\boldsymbol{u})\) and \(w' \in \mathfrak{S}_n\).  Then
\begin{align*}
x &= (x'w')(J_{n+1}^{k}s_n s_{n-1} \cdots s_i) \qquad \text{(Lemma~\ref{Lem:si-Jj}(ii))}\\
&= (x'w')J_{n+1}^{k-1}(s_n J_n + 1)s_{n-1} \cdots s_i\qquad  \text{(Eq~\eqref{Equ:JM-H})}\\
&= (x'w's_{n-1} \cdots s_i)J_{n+1}^{k-1} + (x'w')(J_{n+1}^{k-1}s_n)J_n s_{n-1} \cdots s_i.
\end{align*}
By induction, \(y = J_n s_{n-1} \cdots s_i\) is a $\mathbb{C}(\boldsymbol{u})$--linear combination of words of four types, that is, there exists \(h_{n-1} \in H_{n-1}(\boldsymbol{u})\) such that
\begin{align*}
 y&= h_{n-1}(1 + s_{n-1} \cdots s_i + s_{n-1} \cdots s_i J_i^a + J_n^b)
\end{align*}
for some $1\leq a,b<m-1$. Thus
\begin{align*}(J_{n+1}^{k-1} s_n)y &=  h_{n-1}(J_{n+1}^{k-1} s_n)(1 + s_{n-1} \cdots s_i + s_{n-1} \cdots s_i J_i^a + J_n^b)
\end{align*}
 is a $\mathbb{C}(\boldsymbol{u})$--linear combination of words of four types, showing \(x\) is also such a combination. The standard basis is unique, so any two expressions of $x$ in terms of the four types would imply a non-trivial linear relation in the standard basis, which is impossible. Thus the decomposition is unique.
\end{proof}

\begin{theorem}\label{Them:Basis-for-Markov}
Every element of \(H_{n+1}(\boldsymbol{u})\) can be written uniquely as a $\mathbb{C}(\boldsymbol{u})$--linear combination of words of one of the following types:
\begin{enumerate}
  \item[(a)] \(h_{n}\);
  \item[(b)] \(h_{n}s_n \cdots s_i\), $i = 1, \ldots, n$;
  \item[(c)] \(h_{n}s_n \cdots s_i t_{i}^k\), $i = 1, \ldots, n$, \(1 \leq k\leq m-1\);
  \item[(d)] \(h_{n}t_{n+1}^k\), \(1 \leq k \leq m-1\);
\end{enumerate}
where \(h_{n} \in H_n(\boldsymbol{u})\).
\end{theorem}
\begin{proof}
By Theorem~\ref{Them:Inductive-basis}, we need to express words of types (III) and (IV) in terms of words of type (a)--(d). Let \(h = h_{n}s_n \cdots s_i J_i^k\) (\(1 \leq k\leq m-1\)) be a word of type (III). We prove the assertion holds  by induction on $k$.

For $k = 1$, $J_i=t_i+L_i$ shows
\begin{align*}
 h &= h_{n}s_n \cdots s_i(t_i + L_i) = h_{n}s_n \cdots s_i t_i + h_{n}s_n \cdots s_i L_i,
\end{align*}
which is a combination of types (a), (b), (c). Assume the claim holds for $k = p$. For $k = p+1$, we have
\begin{align*}
  h_{n}s_n \cdots s_i J_i^{p+1}& = h_{n}s_n \cdots s_i J_i^p (t_i + L_i),
\end{align*}
which decomposes into terms handled by induction. Similar argument shows the assertion holds  for words of type (IV).
\end{proof}

\begin{theorem}\label{Them:t-basis}Keeping notations as above, then the set
\begin{align*}
\mathscr{T}_n := \left\{ t_{i_1}^{k_1} \cdots t_{i_r}^{k_r} w \;\middle|\; 1 \leq i_1 < \cdots < i_r \leq n,\; 0 \leq k_1, \ldots, k_r\leq m-1,\; w \in \mathfrak{S}_n \right\}
\end{align*}
forms a \(\mathbb{C}(\boldsymbol{u})\)-basis of \(H_n(\boldsymbol{u})\).
\end{theorem}
\begin{proof}
By Theorem~\ref{Them:Basis-for-Markov}, it suffices to show that the inductive basis (a)--(d) lies in \(\mathscr{T}_n\). Firstly an induction argument shows  $h_{n-1} \in \mathscr{T}_{n-1} \subset \mathscr{T}_n$ when $h_{n-1} \in H_{n-1}(\boldsymbol{u})$.  For any $h_{n-1}\in H_{n-1}(\boldsymbol{u})$,  we assume that  $h_{n-1} = t_{i_1}^{k_1} \cdots t_{i_r}^{k_r} \sigma$ with $\sigma\in \mathfrak{S}_{n-1}$. Then
\begin{eqnarray*}
 && h_{n-1}s_{n-1} \cdots s_i= t_{i_1}^{k_1} \cdots t_{i_r}^{k_r} \sigma s_{n-1} \cdots s_i \in \mathscr{T}_n,\\
 &&h_{n-1}s_{n-1} \cdots s_i t_i^k=t_{i_1}^{k_1} \cdots t_{i_r}^{k_r} \sigma t_{n}^k s_{n-1} \cdots s_i \in \mathscr{T}_n,\\
&&h_{n-1}t_{n+1}^k=t_{i_1}^{k_1} \cdots t_{i_r}^{k_r} \sigma t_{n}^k s_{n-1} \cdots s_i \in \mathscr{T}_n.
\end{eqnarray*}
This proof is completed.
\end{proof}
\section{Normalized Markov traces}\label{Sec:Markov-traces}
In this section, we construct a $\mathbb{C}(\boldsymbol{u})$--linear function on $H_n(\boldsymbol{u})$ and demonstrate it constitutes a normalized Markov trace through Lambropoulou's approach from \cite{L}, thereby proving Theorem~\ref{Them:Markov} from the introduction.

The following lemma is a degeneration version of \cite[Lemma~6]{L}, which can be proved through analogous reasoning:
\begin{lemma}\label{Lemm:Bimodules-iso}Define the map
\begin{eqnarray*}
&&\phi_n:H_n(\boldsymbol{u})\otimes_{H_{n-1}(\boldsymbol{u})}H_{n}(\boldsymbol{u})\bigoplus (\oplus_{i=0}^{m-1}H_{n}(\boldsymbol{u}))\longrightarrow H_{n+1}(\boldsymbol{u}),\end{eqnarray*}
by
\begin{eqnarray*}&&\qquad\qquad a\otimes b\bigoplus (\oplus_{i=0}^{m-1}a_i)\mapsto as_nb+\sum_{i=0}^{m-1}a_kt_{n+1}^i.
\end{eqnarray*}
Then $\phi_n$ is an $H_n(\boldsymbol{u})$-bimodule isomorphism.
\end{lemma}
\begin{proof}Theorem~\ref{Them:Basis-for-Markov} establishes that
  \begin{eqnarray*}
  &&\mathscr{B}_n=\{s_{n-1}\cdots s_i\,|\,1\leq i\leq n-1\}\cup\{t_n^k\,|\,0\leq k\leq m-1\}\\
  &&\qquad\qquad\cup\{s_{n-1}\cdots s_{i}t_i^k\,|\,2\leq i\leq n-2,1\leq k\leq m-1\}
  \end{eqnarray*}
  forms basis for $H_n(\boldsymbol{u})$ as a free $H_{n-1}(\boldsymbol{u})$-module. The universal property of tensor products yields:
  \begin{eqnarray*}
  H_n(\boldsymbol{u})\otimes_{H_{n-1}(\boldsymbol{u})}H_{n}(\boldsymbol{u})&=&H_n(\boldsymbol{u})\otimes_{H_{n-1}(\boldsymbol{u})}
  (\oplus_{\alpha\in\mathscr{B}_n}H_{n-1}(\boldsymbol{u})\cdot\alpha)\\
  &=&\oplus_{\alpha\in\mathscr{B}_n} (H_n(\boldsymbol{u})\otimes_{H_{n-1}(\boldsymbol{u})}H_{n-1}(\boldsymbol{u})\cdot\alpha)\\
  &=&\oplus_{\alpha\in\mathscr{B}_n}H_n(\boldsymbol{u})\cdot\alpha.
                  \end{eqnarray*}
Thus
 \begin{eqnarray*}
  H_n(\boldsymbol{u})\otimes_{H_{n-1}(\boldsymbol{u})}H_{n}(\boldsymbol{u})\bigoplus(\oplus_{i=0}^{m-1}H_n(\boldsymbol{u}))
  &=&\oplus_{\alpha\in\mathscr{B}_n}H_n(\boldsymbol{u})\cdot\alpha\bigoplus(\oplus_{i=0}^{m-1}H_n(\boldsymbol{u})).
                  \end{eqnarray*}
Similarly, the set
  \begin{eqnarray*}
  &&\mathscr{B}_{n+1}=\{s_ns_{n-1}\cdots s_i\,|\,1\leq i\leq n\}
  \cup\{t_{n+1}^k\,|\,0\leq k\leq m-1\}\\
  &&\qquad\qquad\qquad\cup\{s_{n}\cdots s_{i}t_i^k\,|\,1\leq i\leq n-1,1\leq k\leq m-1\}
  \end{eqnarray*}
 forms a basis of $H_{n+1}(\boldsymbol{u})$ as a free $H_{n}(\boldsymbol{u})$-module. Observing that
\begin{eqnarray*}
  &&\mathscr{B}_{n+1}=\{s_n\alpha\,|\,\alpha\in\mathscr{B}_{n}\}\cup\{t_{n+1}^k\,|\,0\leq k\leq m-1\},
  \end{eqnarray*}
  we conclude $\phi_n$ bijectively maps basis elements to $\mathscr{B}_{n+1}$, establishing the  bimodule isomorphism.
\end{proof}

Given $z, y_1, \ldots, y_{m-1}\in \mathbb{C}(\boldsymbol{u})$,
we inductively define a $\mathbb{C}(\boldsymbol{u})$-linear function $\mathrm{tr}$ on $H_{\infty}(\boldsymbol{u})$ as follows:
Given an element $x\in H_{n+1}(\boldsymbol{u})$ expressed via Lemma~\ref{Lemm:Bimodules-iso} as
\begin{equation*}
  x:=\phi_n\left(a\otimes b\oplus\left(\oplus_{k=0}^{m-1}\alpha_kt_{n+1}^k\right)\right),
\end{equation*}
 we define
\begin{eqnarray}\label{Equ:tr-def}
&& \mathrm{tr}(x):=z\mathrm{\,tr}(ab)+\mathrm{tr}(\alpha_0)+\sum_{k=1}^{m-1}
y_k\mathrm{tr}(\alpha_k),
\end{eqnarray}
where $y_k=\mathrm{tr}(t_{n+1}^k)$ (independent of $n$) and $\mathrm{tr}(1)=1$. This construction satisfies properties (m1), (m3), and (m4) of Theorem~\ref{Them:Markov} in Introduction.

To show that the $\mathbb{C}(\boldsymbol{u})$-linear function $\mathrm{tr}$ defined by Eq.~(\ref{Equ:tr-def}) satisfies  $\mathrm{tr}(\alpha\beta)=\mathrm{tr}(\beta\alpha)$ for all $\alpha,\beta\in H_{\infty}(\boldsymbol{u})$, we need the following lemmas.

\begin{lemma}\label{Lemm:xs_nys_n}For $x,y\in H_{n}(\boldsymbol{u})$,
\begin{align*}\mathrm{tr}(xs_nys_n)=\mathrm{tr}(s_nxs_ny).
  \end{align*}
In particular,  \begin{align*}
 \mathrm{tr}(t_n^ks_ns_{n-1}s_n)=\mathrm{tr}(s_nt_n^{k}s_ns_{n-1}).
  \end{align*}
 \end{lemma}
\begin{proof}Clearly the second equality follows directly from the first one by letting $x=t_n^k$ and $y=s_{n-1}$. We use the case analysis to establish the first equality:

\textbf{Case 1:} $x,y\in H_{n-1}(\boldsymbol{u})$. Then $s_nx=xs_n$ and $s_ny=ys_n$, which implies
\begin{align*}\mathrm{tr}(xs_nys_n)=\mathrm{tr}(xy)=\mathrm{tr}(s_nxs_ny).
  \end{align*}

  \textbf{Case 2:} $x\in H_{n-1}(\boldsymbol{u})$ and $y=\alpha s_{n-1}\beta$ with $\alpha,\beta\in H_{n-1}(\boldsymbol{u})$, or vice versa. Then $s_nx=xs_n$, $s_n\alpha=\alpha s_n$, and $s_n\beta=\beta s_n$, which show
  \begin{align*}
   \mathrm{tr}(xs_nys_n)&=\mathrm{tr}(xs_n\alpha s_{n-1}\beta s_n)=\mathrm{tr}(x \alpha s_{n-1}s_{n}s_{n-1}\beta)=z\mathrm{tr}(x\alpha\beta),\\
   \mathrm{tr}(s_nxs_ny)&=\mathrm{tr}(xy)=\mathrm{tr}(x\alpha s_{n-1}\beta)=z\mathrm{tr}(x\alpha\beta).
  \end{align*}

  \textbf{Case 3:} $x=\alpha s_{n-1}\beta$ and $y=\delta s_{n-1}\gamma$ with  $\alpha,\beta, \delta,\gamma\in H_{n-1}(\boldsymbol{u})$. Then
  \begin{align*}
   \mathrm{tr}(\alpha s_{n-1}\beta s_n \delta s_{n-1}\gamma s_n)&=\mathrm{tr}(\alpha s_{n-1}\beta  \delta (s_ns_{n-1}s_n)\gamma )\\&=z\mathrm{tr}(\alpha s_{n-1}\beta \delta\gamma)\\&=z^2\mathrm{tr}(\alpha\beta\delta\gamma),\\
   \mathrm{tr}(s_n\alpha s_{n-1}\beta s_n \delta s_{n-1}\gamma)&=\mathrm{tr}(\alpha (s_ns_{n-1}s_n\beta\delta s_{n-1}\gamma)\\&=z\mathrm{tr}(\alpha\beta\delta s_{n-1}\gamma)\\&=z^2\mathrm{tr}(\alpha\beta\delta \gamma).
  \end{align*}

 \textbf{Case 4:} $x=\alpha t_{n}^{\ell}$ and $y=\beta t_n^{k}$ with $\alpha, \beta\in H_{n-1}(\boldsymbol{u})$, $0\leq \ell,k<m$. Then
 \begin{align*}
   \mathrm{tr}(\alpha t_{n}^{\ell}s_n\beta t_n^ks_n)&=\mathrm{tr}(\alpha t_{n}^{\ell}\beta (s_nt_n^ks_n))=y_k\mathrm{tr}(\alpha t_{n}^{\ell}\beta)=y_ky_{\ell}\mathrm{tr}(\alpha\beta),\\
   \mathrm{tr}(s_n\alpha t_{n}^{\ell}s_n\beta t_{n}^k)&=\mathrm{tr}(\alpha (s_nt_n^{\ell}s_n)\beta t_n^{k})=y_{\ell}\mathrm{tr}(\alpha\beta t_n^{k})=y_ky_{\ell}\mathrm{tr}(\alpha\beta).
  \end{align*}

  \textbf{Case 5:} $x=\alpha s_{n-1}\beta$ and $y=\gamma t_n^{\ell}$ with $\alpha,\beta, \gamma\in H_{n-1}(\boldsymbol{u})$ and $1\leq \ell<m$, or vice versa. Then
  \begin{align*}
   \mathrm{tr}(\alpha s_{n-1}\beta s_n\gamma t_{n}^{\ell}s_n)&=\mathrm{tr}(\alpha s_{n-1}\beta \gamma (s_nt_{n}^{\ell}s_n))\\&=y_{\ell}\mathrm{tr}(\alpha s_{n-1}\beta \gamma)\\&=zy_{\ell}\mathrm{tr}(\alpha\beta\gamma),\\
   \mathrm{tr}(s_n\alpha s_{n-1}\beta s_n\gamma t_{n}^{\ell})&=\mathrm{tr}(\alpha (s_n s_{n-1}s_n)\beta\gamma t_n^{\ell})\\&=z\mathrm{tr}(\alpha\beta\gamma t_n^{\ell})\\&=zy_{\ell}\mathrm{tr}(\alpha\beta\gamma).
  \end{align*}
  We complete the proof.
\end{proof}

\begin{lemma}\label{Lemm:Type-II}For $h_{n}\in H_n(\boldsymbol{u})$  and $1\leq i\leq n-1$, we have
\begin{eqnarray*}
 &&\mathrm{tr}(h_{n}s_ns_{n-1}\cdots s_it)=\mathrm{tr}(th_{n}s_ns_{n-1}\cdots s_i)\\
 &&\mathrm{tr}(h_{n}s_ns_{n-1}\cdots s_is_j)=\mathrm{tr}(s_jh_{n}s_ns_{n-1}\cdots s_i) \text{ for }j=1, \ldots, n. \end{eqnarray*}
   \end{lemma}

\begin{proof} By an induction argument, we obtain
\begin{align*}
\mathrm{tr}(h_{n-1}s_ns_{n-1}\cdots s_it)&=z\mathrm{\,tr}(h_{n}s_{n-1}\cdots s_it)\\
 &=z\mathrm{\,tr}(th_{n}s_{n-1}\cdots s_i)\\
 &=\mathrm{tr}(th_{n}s_n\cdots s_i). \end{align*}
 For $j=1, \ldots, n-1$, we have
 \begin{align*}
 \mathrm{tr}(h_{n}s_n\cdots s_is_j)&=z\mathrm{\,tr}(h_{n}s_{n-1}\cdots s_is_j)\\
 &=z\mathrm{\,tr}(s_jh_{n}s_{n-1}\cdots s_i)\\
 &=\mathrm{tr}(s_jh_{n}s_n\cdots s_i), \end{align*}
 where the first equality and last one follow by using Theorem~\ref{Them:Markov}~(m3). For $j=n$, the second equality follows directly by applying Lemma~\ref{Lemm:xs_nys_n}.
 \end{proof}

\begin{lemma}\label{Lemm:Typer-III}For $h_{n}\in H_{n}(\boldsymbol{u})$ and $1\leq i\leq n-1$,  \begin{align*}
\mathrm{tr}(h_{n}s_ns_{n-1}\cdots s_{i}t_{i}^ks_n)&=
\mathrm{tr}(s_nh_{n}s_ns_{n-1}\cdots s_{i}t_{i}^k),\\
\mathrm{tr}(h_{n}s_ns_{n-1}\cdots s_{i}t_{i}^kt)&=
\mathrm{tr}(th_{n}s_ns_{n-1}\cdots s_{i}t_{i}^k).
\end{align*}
\end{lemma}
\begin{proof}Clearly Lemma~\ref{Lemm:xs_nys_n} implies the first equality and the second equality holds for $n=1$. Assume that the second equality holds for $n\geq 1$. We show it holds for $n+1$. By applying rule (m3), we obtain
\begin{align*}
\mathrm{tr}(h_{n}s_ns_{n-1}\cdots s_{i}t_{i}^kt)&=z\,
\mathrm{tr}(h_{n}s_{n-1}\cdots s_{i}t_{i}^kt)\\
&=z\,\mathrm{tr}(th_{n}s_{n-1}\cdots s_{i}t_{i}^k)\\
&=\mathrm{tr}(th_{n}s_ns_{n-1}\cdots s_{i}t_{i}^k).
\end{align*}
Thus the proof is completed.
\end{proof}

\begin{lemma}\label{Lemm:Type-IV}For $h_{n}\in H_{n}(\boldsymbol{u})$ and $1\leq k\leq m-1$,
\begin{align*}
\mathrm{tr}(h_{n}t_{n+1}^kt)&=\mathrm{tr}(th_{n}t_{n+1}^k),\\
\mathrm{tr}(h_{n}t_{n+1}^ks_i)&=\mathrm{tr}(s_ih_{n}t_{n+1}^k) \text{ for }i=1, \ldots, n.
\end{align*}
\end{lemma}
\begin{proof}Note that $s_{n,1}s_{2,n}=s_{1,n-1}s_ns_{n-1,1}$. Thanks to Lemma~\ref{Lemm:s-t}(vi) and Eq.~\eqref{Equ:tr-def}, we get
\begin{align*}
\mathrm{tr}(h_{n}t_{n+1}^kt)&=\mathrm{tr}(h_{n}tt_{n+1}^k)+\mathrm{tr}(h_{n}t^ks_{n,1}s_{2,n})
-\mathrm{tr}(h_{n}s_{n,2}s_{1,n}t^k)\\
&=y_k\mathrm{tr}(h_{n}t).
\end{align*}
On the other hand, Eq.~\eqref{Equ:tr-def} shows
\begin{align*}
\mathrm{tr}(th_{n}t_{n+1}^k)&=y_k\mathrm{tr}(th_{n}).
\end{align*}
Finally the induction argument shows $\mathrm{tr}(th_{n})=\mathrm{tr}(h_{n}t)$. Thus  the first equality holds.

For $i=1, \ldots, n-1$, we have
\begin{align*}
 t_{n+1}^ks_i&=s_n\cdots s_1t^ks_1\cdots s_{i-1}(s_is_{i+1}s_i)s_{i+2}\cdots s_n \\
 &=s_n\cdots s_1t^ks_1\cdots s_{i-1}s_{i+1}s_is_{i+1}s_{i+2}\cdots s_n\\
 &=s_n\cdots s_{i+2} (s_{i+1}s_is_{i+1})s_{i-1}\cdots s_1t^ks_1\cdots  s_n\\
 &=s_it_{n+1}^k.
\end{align*}
 Then Eq.~\eqref{Equ:tr-def} implies \begin{align*}
\mathrm{tr}(h_{n}t_{n+1}^ks_i)&=\mathrm{tr}(h_{n}s_it_{n+1}^k)=y_k\mathrm{tr}(h_{n}s_i).
\end{align*}
While $\mathrm{tr}(s_ih_{n}t_{n+1}^k)=y_k\mathrm{tr}(s_ih_{n})$. Thus the second equality holds for $i=1, \ldots,n-1$ by the induction argument.

For $i=n$, clearly
\begin{align*}
 \mathrm{tr}(h_{n}t_{n+1}^ks_n)=\mathrm{tr}(h_{n}s_{n}t_n^{k})=z\,\mathrm{tr}(h_{n}t_n^{k}).
\end{align*}
 Case by case calculation shows: for  $h_{n}\in H_{n-1}(\boldsymbol{u})$,
 \begin{align*}
  \mathrm{tr}(s_nh_{n}t_{n+1}^k)=\mathrm{tr}(h_{n}t_{n}^ks_n)=z\,\mathrm{tr}(h_{n}t_{n}^k);
 \end{align*}
For $h_{n}=\alpha s_{n-1}\beta\in H_{n}(\boldsymbol{u})$ where $\alpha,\beta\in H_{n-1}(\boldsymbol{u})$,
\begin{align*}
\mathrm{tr}(s_nh_{n}t_{n+1}^k)&=
\mathrm{tr}((\alpha  s_{n-1})s_n(s_{n-1}\beta t_{n}^k)s_n)\\
&=\mathrm{tr}(s_n(\alpha  s_{n-1})s_n(s_{n-1}\beta t_{n}^k)) \quad \text{(Lemma~\ref{Lemm:xs_nys_n})}\\
&=\mathrm{tr}(\alpha s_{n-1}s_n\beta t_{n}^k)\\
&=z\,\mathrm{tr}(h_{n}t_{n}^k),
\end{align*}
where the last equality follows by applying Eq.~\eqref{Equ:tr-def};

 For $h_{n}=\alpha t_{n}^{\ell}\in H_{n}(\boldsymbol{u})$ where $\alpha\in H_{n-1}(\boldsymbol{u})$ and $1\leq \ell\leq m-1$,
\begin{align*}
\mathrm{tr}(s_n\alpha t_n^{\ell} t_{n+1}^k)
&=\mathrm{tr}(\alpha s_n t_n^{\ell}t_{n+1}^k)\\
&=\mathrm{tr}(\alpha t_n^{k}s_nt_n^{\ell})+\sum_{i=1}^{\ell}\mathrm{tr}(\alpha t_{n+1}^{\ell-i}t_{n}^{k-i-2})-\sum_{i=1}^{\ell}\mathrm{tr}(\alpha t_{n+1}^{k-i-2}t_{n}^{\ell-i})\\
&=z\,\mathrm{tr}(h_{n}t_{n}^k),
\end{align*}
where the second equality follows by applying Corollary~\ref{Cor:t-n-n+1} and the last equality follows by applying Eq.~\eqref{Equ:tr-def}.  \end{proof}

Now we can show that the linear function $\mathrm{tr}$ defined by Eq.~(\ref{Equ:tr-def}) is a trace on $H_{\infty}(\boldsymbol{u})$.

\begin{theorem}\label{Them:alpha-beta=beta-alpha}For all $\alpha,\beta\in H_{\infty}(\boldsymbol{u})$,
\begin{align*}\mathrm{tr}(\alpha\beta)=\mathrm{tr}(\beta\alpha).
\end{align*}
\end{theorem}

\begin{proof}By induction argument, it suffices to show that
\begin{align*}
  \mathrm{tr}(\alpha s_i)=\mathrm{tr}(s_i\alpha)&\text{ and }
  \mathrm{tr}(\alpha t)=\mathrm{tr}(t\alpha)
\end{align*}
for all $\alpha\in H_{n+1}(\boldsymbol{u})$ and $i=1,\ldots,n$.
Thanks to Theorem~\ref{Them:Basis-for-Markov}, if $\alpha$ is of type (a),  the induction step implies $\mathrm{tr}(\alpha t)=\mathrm{tr}(t\alpha)$ and $\mathrm{tr}(\alpha s_i)=\mathrm{tr}(s_i\alpha)$ for $i=1, \ldots, n-1$. Further Eq.~\eqref{Equ:tr-def} shows
\begin{equation*}
  \mathrm{tr}(\alpha s_n)=z\mathrm{\,tr}(\alpha)=\mathrm{tr}(s_n\alpha).
\end{equation*}
For $\alpha$ being of type (b), (c) and (d), the assertion follows by applying Lemmas~\ref{Lemm:Type-II}, \ref{Lemm:Typer-III} and  \ref{Lemm:Type-IV} respectively.
\end{proof}

Now we are ready to prove Theorem~\ref{Them:Markov} in the introduction.
\begin{proof}[Proof of Theorem~\ref{Them:Markov}] (m1), (m3), (m4) follow from its construction (see Eq.~\eqref{Equ:tr-def}, while (m2) derives from Theorem~\ref{Them:alpha-beta=beta-alpha}.  Uniqueness follows from Theorem~\ref{Them:Basis-for-Markov}.\end{proof}

Let us remark that the uniqueness of the normalized Markov trace $\mathrm{tr}$ can be also proved by applying Geck and Lambropoulou's argument in \cite[Lemma~4.2]{GL}.  More precisely, we have the following fact.

\begin{lemma}Let $\tau: H_n(\boldsymbol{u})\rightarrow\mathbb{C}(\boldsymbol{u},z)$ be a normalized Markov trace.  If $n\geq 1$, $k\geq 1$, $h\in H_{n}$, and $1 \leq i_2,\ldots, i_{k}\leq m-1$, then
   \begin{align*}
     \tau(hs_{n}t_{n+2}^{i_2}\cdots t_{n+k}^{i_k})&=z\tau(ht_{n+1}^{i_2}t_{n+2}^{i_3}\cdots t_{n+k-1}^{i_k})
   \end{align*}
   and
   \begin{align*}
     \tau(ht_{n+2}^{i_2}\cdots t_{n+k}^{i_k})&=\tau(ht_{n+1}^{i_2}t_{n+2}^{i_3}\cdots t_{n+k-1}^{i_k}).
   \end{align*}
   \end{lemma}

\begin{proof}We prove the first equality by an induction argument on $k$. For $k=1$, it follows directly by applying Definition~\ref{Def:Markov-traces}. Now we assume that $k>1$. Then we have
   \begin{align*}
     \tau(hs_{n}t_{n+2}^{i_2}\cdots t_{n+k}^{i_k})&=\tau(hs_{n}s_{n+1}t_{n+1}^{i_2}s_{n+1}t_{n+3}^{i_3}\cdots t_{n+k}^{i_k})&\mathrm{(}t_{n+2}=s_{n+1}t_{n+1}s_{n+1}\mathrm{)}\\
     &=\tau(hs_{n}s_{n+1}t_{n+1}^{i_2}t_{n+3}^{i_3}\cdots t_{n+k}^{i_k}s_{n+1})
     &(\text{The proof of Lemma}~\ref{Lemm:Type-IV})\\
     &=\tau(s_{n+1}hs_{n}s_{n+1}t_{n+1}^{i_2}t_{n+3}^{i_3}\cdots t_{n+k}^{i_k})&(\tau \text{ is a trace})\\
     &=\tau(hs_{n+1}s_{n}s_{n+1}t_{n+1}^{i_2}t_{n+3}^{i_3}\cdots t_{n+k}^{i_k})&(hs_{n+1}=s_{n+1}h)\\
     &=\tau(hs_ns_{n+1}s_{n}t_{n+1}^{i_2}t_{n+3}^{i_3}\cdots t_{n+k}^{i_k})&(\text{Braid relations})\\
     &=\tau(hs_nt_{n}^{i_2}s_{n+1}s_nt_{n+3}^{i_3}\cdots t_{n+k}^{i_k})&(s_{n+1}t_n=t_ns_{n+1})\\
     &=\tau(hs_nt_{n}^{i_2}s_{n+1}t_{n+3}^{i_3}\cdots t_{n+k}^{i_k}s_n)&(\text{The proof of Lemma}~\ref{Lemm:Type-IV})\\
     &=\tau((s_nhs_nt_{n}^{i_2})s_{n+1}t_{n+3}^{i_3}\cdots t_{n+k}^{i_k})&(\tau \text{ is a trace})\\
     &=z\tau(s_nhs_nt_{n}^{i_2}t_{n+2}^{i_3}\cdots t_{n+k-1}^{i_k})&(\text{Induction})\\
     &=z\tau(hs_nt_{n}^{i_2}t_{n+2}^{i_3}\cdots t_{n+k-1}^{i_k}s_n)&(\tau \text{ is a trace})\\
     &=z\tau(hs_nt_{n}^{i_2}s_nt_{n+2}^{i_3}\cdots t_{n+k-1}^{i_k})&(\text{The proof of Lemma}~\ref{Lemm:Type-IV})\\
     &=z\tau(ht_{n+1}^{i_2}t_{n+2}^{i_3}\cdots t_{n+k-1}^{i_k}).
   \end{align*}
   The second can be proved by an analogue computation.
\end{proof}

\section{Non-normalized Markov traces}\label{Sec:BK-trace}
In this section we  construct a $\mathbb{C}(\boldsymbol{u})$--linear function on $H_n(\boldsymbol{u})$ by using the standard bases and show it constitutes a non-normalized Markov trace, thereby proving Theorem~\ref{Them:Markov-BK} from the introduction.

 The following easy verified fact will be useful.
\begin{lemma}\label{Lemm:J-s_n}
For integers $1 \leq k\leq m-1$ and $n \geq 1$, we have
\begin{align*}
J_{n+1}^k s_n &= s_n J_n^k + \sum_{i=0}^{k-1} J_{n+1}^{k-1-i} J_n^i.
\end{align*}
\end{lemma}
\begin{proof}Apply induction  on $k$ and $k=1$ follows by applying Eq.~\eqref{Equ:JM-H}. Assume that it holds for $k-1\geq 1$. For $k$,
  \begin{align*}
J_{n+1}^{k} s_n &= J_{n+1}^{k-1}(s_nJ_n+1)\\
&= \biggl(s_n J_n^{k-1} + \sum_{i=0}^{k-2} J_{n+1}^{k-2-i} J_n^i\biggr)J_n+J_{n+1}^{k-1}\\
&=s_nJ_n^{k}+\sum_{i=0}^{k-2} J_{n+1}^{k-2-i} J_n^{i+1}+J_{n+1}^{k-1}\\
&= s_n J_n^k + \sum_{i=0}^{k-1} J_{n+1}^{k-1-i} J_n^i.
\end{align*}
The proof is completed.
\end{proof}
\begin{lemma}\label{Lemm:Bimodules-iso-JM}
Define the bimodule map
\[\psi_n: H_n(\boldsymbol{u}) \otimes_{H_{n-1}(\boldsymbol{u})} H_n(\boldsymbol{u}) \oplus \biggl(\bigoplus_{i=0}^{m-1} H_n(\boldsymbol{u})\biggr) \to H_{n+1}(\boldsymbol{u})\]
by
\[a \otimes b \oplus \biggl(\bigoplus_{i=0}^{m-1} a_i\biggr) \mapsto a s_n b + \sum_{i=0}^{m-1} a_i J_{n+1}^i.\]
Then $\psi_n$ is an $H_n(\boldsymbol{u})$-bimodule isomorphism.
\end{lemma}
\begin{proof}
Analogous to Lemma~\ref{Lemm:Bimodules-iso}'s proof using standard basis arguments.
\end{proof}

Given $z, y_1, \ldots, y_{m-1}\in \mathbb{C}(\boldsymbol{u})$, we inductively define a $\mathbb{C}(\boldsymbol{u})$-linear function $\mathrm{Tr}$ on $H_\infty(\boldsymbol{u})$ as follows. Assume $\mathrm{Tr}$ is defined on $H_n(\boldsymbol{u})$. For $x \in H_{n+1}(\boldsymbol{u})$ decomposed via Lemma~\ref{Lemm:Bimodules-iso-JM} as
\[x = \psi_n\biggl(a \otimes b \oplus \biggl(\bigoplus_{k=0}^{m-1} \alpha_k J_{n+1}^k\biggr)\biggr),\]
we define
\begin{equation}\label{Equ:tr-def-JM}
\mathrm{Tr}(x) :=z\cdot\mathrm{Tr}(ab) + \mathrm{Tr}(\alpha_0) + \sum_{k=1}^{m-1} \mathrm{Tr}(J_{n+1}^k) \mathrm{Tr}(\alpha_k),
\end{equation}
where $\mathrm{Tr}(1) = 0$ and $\mathrm{Tr}(J_{n+1}^k)$ determined by:
\begin{itemize}
\item[(M4)] $\mathrm{Tr}(h s_n J_{n}^k s_n) = \mathrm{Tr}(J_{n}^k) \mathrm{Tr}(h)$ for all $h \in H_n(\boldsymbol{u})$ and for all $1 \leq k \leq m-1$;
\item[(M5)] $\mathrm{Tr}(J_1^k) =y_k$ for all $k=1,\ldots,m-1$.
\end{itemize}
Let us remark that (M4) and Lemma~\ref{Lemm:J-s_n} show that
\begin{equation*}
  \mathrm{Tr}(J_{n+1}^k)=\mathrm{Tr}(J_{n}^k)+\sum_{i=0}^{k-1} \mathrm{Tr}(J_n^iJ_{n+1}^{k-1-i}s_n)
\end{equation*}
for $1\leq k\leq m-1$, $n=1,2,\ldots$.

\begin{proposition}\label{Prop:Tr-M}
The $\mathbb{C}(\boldsymbol{u})$--linear function $\mathrm{Tr}$ satisfies
\begin{enumerate}
\item[(M1)] $\mathrm{Tr}(s_n) = z$ for all $n \geq 1$;
\item[(M3)] $\mathrm{Tr}(x s_n y) = z\cdot\mathrm{Tr}(xy)$ for all $x,y \in H_n(\boldsymbol{u})$;
\item[(M4')] $\mathrm{Tr}(s_n J_n^k s_n) = \mathrm{Tr}(J_n^k)$ for  $1 \leq k \leq m-1$ and for all $n\geq 1$.
\end{enumerate}
\end{proposition}
\begin{proof}(M1, M3) follows directly from Eq.~\eqref{Equ:tr-def-JM} via bimodule structure, while (M4') is immediate from (M4).
\end{proof}

\begin{remark}
This trace construction specializes to the Brundan--Kleshchev trace $\tau_{\mathrm{BK}}$ through parameter specialization, analogous to \cite[Lemma~4.3]{GIM}. The non-normalized condition $\mathrm{Tr}(1)=0$ distinguishes it from the normalized Markov traces while preserving essential conjugation properties.
\end{remark}

We now show that the $\mathbb{C}(\boldsymbol{u})$--linear function $\mathrm{Tr}$ defined by Eq.~\eqref{Equ:tr-def-JM} is a trace on $H_{\infty}(\boldsymbol{u})$, i.e.,
\begin{equation*}
    \mathrm{Tr}(\alpha\beta) = \mathrm{Tr}(\beta\alpha) \quad \text{ for all } \alpha,\beta \in H_{\infty}(\boldsymbol{u}).
\end{equation*}
Our approach follows arguments parallel to those in Section~\ref{Sec:Markov-traces}.

\begin{lemma}\label{Lemm:tr(xJy)=tr(xy)}
For any $x,y \in H_n(\boldsymbol{u})$ and $1 \leq k\leq m-1$,
\begin{equation*}
    \mathrm{Tr}(xJ_{n+1}^k y) = \mathrm{Tr}(J_{n+1}^k) \mathrm{Tr}(xy).
\end{equation*}
\end{lemma}

\begin{proof}
Let $w \in \mathfrak{S}_n$ be expressed in Jones normal form:
\begin{align*}
    w = (s_{i_1}s_{i_1-1}\cdots s_{k_1})(s_{i_2}s_{i_2-1}\cdots s_{k_2}) \cdots (s_{i_r}s_{i_r-1}\cdots s_{k_r}),
\end{align*}
where $0 < i_1 < \cdots < i_r < n$, $0 < k_1 < \cdots < k_r \leq n-1$, and $i_j \geq k_j$ for all $j$. By Eq.~\eqref{Equ:Standard-basis}, we may assume that $y = J_1^{a_1}\cdots J_n^{a_n}w' s_{n-1}$, where $0 \leq a_i\leq m-1$ and $w' \in \mathfrak{S}_{n-1}$.

Applying Lemma~\ref{Lem:si-Jj}(ii, iii) for commutation relations:
\begin{align*}
    \mathrm{Tr}(xJ_{n+1}^k y) &= \mathrm{Tr}(xJ_{n+1}^k J_1^{a_1}\cdots J_n^{a_n}w's_{n-1}) \\
    &= \mathrm{Tr}(xJ_1^{a_1}\cdots J_n^{a_n} w's_{n-1} J_{n+1}^{k}) \quad \text{(Lemma~\ref{Lem:si-Jj})} \\
    &= \mathrm{Tr}(xy) \mathrm{Tr}(J_{n+1}^k). 
\end{align*}
It completes the proof.
\end{proof}

\begin{lemma}\label{Lemm:JM-xs_nys_n}
For sll $x,y\in H_{n}(\boldsymbol{u})$,
\begin{equation*}
    \mathrm{Tr}(xs_nys_n) = \mathrm{Tr}(s_nxs_ny).
\end{equation*}
In particular, for all $1\leq k\leq m-1$,
\begin{equation*}
    \mathrm{Tr}(J_n^ks_ns_{n-1}s_n) = \mathrm{Tr}(s_nJ_n^{k}s_ns_{n-1}).
\end{equation*}
\end{lemma}

\begin{proof}
The second equality follows from the first by taking $x = J_n^k$ and $y = s_{n-1}$. We establish the first equality through case by case analysis:

\noindent\textbf{Case 1.} $x,y \in H_{n-1}(\boldsymbol{u})$:
\begin{align*}
    \mathrm{Tr}(xs_nys_n) &= \mathrm{Tr}(xy) = \mathrm{Tr}(s_nxs_ny).
\end{align*}
\noindent\textbf{Case 2.} $x \in H_{n-1}(\boldsymbol{u})$, $y = \alpha s_{n-1}\beta$ where $\alpha,\beta\in H_{n-1}(\boldsymbol{u})$, or vice versa:
\begin{align*}
    \mathrm{Tr}(xs_nys_n) &=\mathrm{Tr}(x\alpha(s_ns_{n-1}s_n)\beta)=\mathrm{Tr}(s_nxs_ny).
\end{align*}
\noindent\textbf{Case 3.} $x = \alpha s_{n-1}\beta$, $y = \delta s_{n-1}\gamma$ where $\alpha,\beta,\delta,\gamma\in H_{n-1}(\boldsymbol{u})$:
\begin{align*}
    \mathrm{Tr}(xs_nys_n) &=\mathrm{Tr}(\alpha s_{n-1}\beta\delta(s_ns_{n-1}s_n)\gamma)= z^2\,\mathrm{Tr}(\alpha\beta\delta\gamma), \\
    \mathrm{Tr}(s_nxs_ny) &=\mathrm{Tr}(\alpha (s_ns_{n-1}s_n)\beta\delta s_{n-1}\gamma)= z^2\,\mathrm{Tr}(\alpha\beta\delta\gamma).
\end{align*}
\noindent\textbf{Case 4.} $x \in H_{n-1}(\boldsymbol{u})$, $y = \alpha J_n^{\ell}$ where $\alpha\in H_{n-1}(\boldsymbol{u})$ and $1\leq\ell\leq m-1$, or vice versa:
\begin{align*}
    \mathrm{Tr}(xs_nys_n)&=\mathrm{Tr}(x\alpha s_nJ_n^{\ell}s_n)=
    \mathrm{Tr}(J_n^{\ell})\mathrm{Tr}(x\alpha), \\
    \mathrm{Tr}(s_nxs_ny) &=\mathrm{Tr}(x\alpha J_n^{\ell})= \mathrm{Tr}(J_n^{\ell})\mathrm{Tr}(x\alpha).
\end{align*}
\noindent\textbf{Case 5.} $x = \alpha J_n^{\ell}$, $y = \beta J_n^k$ where $\alpha,\beta\in H_{n-1}(\boldsymbol{u})$ and $1\leq k, \ell\leq m-1$:
\begin{align*}
 \mathrm{Tr}(xs_nys_n)&= \mathrm{Tr}(\alpha\beta J_n^{\ell}s_nJ_n^ks_n)=\mathrm{tr}(J_n^{k})\mathrm{tr}(J_n^{\ell})\mathrm{Tr}(\alpha\beta), \\
\mathrm{Tr}(s_nxs_ny)&= \mathrm{Tr}(\alpha s_nJ_n^{\ell}s_nJ_n^k\beta)=\mathrm{tr}(J_n^{\ell})\mathrm{Tr}(J_n^k)\mathrm{Tr}(\alpha\beta).
\end{align*}

\noindent\textbf{Case 6.} $x = \alpha s_{n-1}\beta$, $y = \gamma J_n^{\ell}$ where $\alpha,\beta,\gamma\in H_{n-1}(\boldsymbol{u})$, or vice versa:
\begin{align*}
    \mathrm{Tr}(xs_nys_n) &=\mathrm{Tr}(\alpha s_{n-1}\beta\gamma s_n J_n^{\ell}s_n)= z\,\mathrm{Tr}(J_n^{\ell})\mathrm{Tr}(\alpha\beta\gamma), \\
    \mathrm{Tr}(s_nxs_ny) &=\mathrm{Tr}(\alpha s_ns_{n-1}s_n\beta\gamma J_n^{\ell})= z\,\mathrm{Tr}(J_n^{\ell})\mathrm{Tr}(\alpha\beta\gamma).
\end{align*}
All cases demonstrate the required equality, thus completing the proof.
\end{proof}

\begin{lemma}\label{Lemm:JM-Typer-III}
For any $h_n \in H_n(\boldsymbol{u})$, $1 \leq i\leq n-1$, and $1 \leq k\leq m-1$,
\begin{align*}
    \mathrm{Tr}(h_n s_n s_{n-1} \cdots s_i J_i^k s_n) &=
    \mathrm{Tr}(s_n h_n s_n s_{n-1} \cdots s_i J_i^k), \\
    \mathrm{Tr}(h_n s_n s_{n-1} \cdots s_i J_i^k t) &=
     \mathrm{Tr}(t h_n s_n s_{n-1} \cdots s_i J_i^k).
\end{align*}
\end{lemma}
\begin{proof}
For the \textbf{first equality}, we have
\begin{align*}
    \mathrm{Tr}(h_n s_n s_{n-1} \cdots s_i J_i^k s_n)
    &= \mathrm{Tr}(h_n s_n s_{n-1} s_n s_{n-2} \cdots s_i J_i^k) \\
    &= z\,\mathrm{Tr}(h_n s_{n-2} \cdots s_i J_i^k).
\end{align*}
For \(\mathrm{Tr}(s_n h_n s_n s_{n-1} \cdots s_i J_i^k)\), we apply the case by case analysis:

\underline{Case 1.} \(h_n \in H_{n-1}(\boldsymbol{u})\):
\begin{align*}
    \mathrm{Tr}(s_n h_n s_n s_{n-1} \cdots s_i J_i^k)
    &= \mathrm{Tr}(h_n s_{n-1} \cdots s_i J_i^k) \\
    &= z\,\mathrm{Tr}(h_n s_{n-2} \cdots s_i J_i^k).
\end{align*}

\underline{Case 2.} \( h_n = \alpha s_{n-1} \beta \) where \( \alpha, \beta \in H_{n-1}(\boldsymbol{u}) \):
\begin{align*}
    \mathrm{Tr}(s_n h_n s_n \cdots s_i J_i^k)
    &= \mathrm{Tr}(\alpha s_n s_{n-1} s_n \beta s_{n-1} \cdots s_i J_i^k) \\
    &= z^2\,\mathrm{Tr}(\alpha \beta s_{n-2} \cdots s_i J_i^k) \\
    &= z\,\mathrm{Tr}(h_n s_{n-2} \cdots s_i J_i^k).
\end{align*}

\underline{Case 3.} $h_n = \alpha J_n^\ell$ where
$\alpha \in H_{n-1}(\boldsymbol{u}), 1 \leq \ell\leq m -1$:
\begin{align*}
    \mathrm{Tr}(s_n h_n s_n \cdots s_i J_i^k)
    &= \mathrm{Tr}(\alpha (s_n J_n^\ell s_n) s_{n-1} \cdots s_i J_i^k) \\
    &= z\,\mathrm{Tr}(J_n^\ell) \mathrm{Tr}(\alpha s_{n-2} \cdots s_i J_i^k) \\
    &= z\,\mathrm{Tr}(h_n s_{n-2} \cdots s_i J_i^k).
\end{align*}
Thus, the first equality holds.

\textbf{Second Equality:} By induction on $n$. It is trivial for $n = 1$. Assume it is true for $n \geq 1$. For $n+1$, we have
\begin{align*}
    \mathrm{Tr}(h_{n+1} s_{n+1} \cdots s_i J_i^k t)
    &= z\,\mathrm{Tr}(h_{n+1} s_{n} \cdots s_i J_i^k t) \quad \text{(by (M3))} \\
    &= z\,\mathrm{Tr}(t h_{n+1} s_{n} \cdots s_i J_i^k) \quad \text{(inductive hypothesis)} \\
    &= \mathrm{Tr}(t h_{n+1} s_{n+1} \cdots s_i J_i^k).
\end{align*}
This completes the proof.
\end{proof}
\begin{lemma}\label{Lemm:JM-Type-IV}For any $h_{n}\in H_{n}(\boldsymbol{u})$, $i=1, \ldots, n$, and $1\leq k\leq m-1$, we have
\begin{eqnarray*}
\mathrm{Tr}(h_{n}J_{n+1}^kt)&=&\mathrm{Tr}(th_{n}J_{n+1}^k),\\
\mathrm{Tr}(h_{n}J_{n+1}^ks_i)&=&\mathrm{Tr}(s_ih_{n}J_{n+1}^k).
\end{eqnarray*}
\end{lemma}
\begin{proof}Thanks to Lemma~\ref{Lemm:s-t}(vi),
\begin{align*}
\mathrm{Tr}(h_{n}J_{n+1}^kt)&=\mathrm{Tr}(J_{n+1}^k)\mathrm{Tr}(h_{n}t)=\mathrm{Tr}(J_{n+1}^k)
\mathrm{Tr}(th_{n})=\mathrm{Tr}(th_{n}J_{n+1}^k),
\end{align*}
where the second equality follows by the induction argument on $n$.  Thus  the first equality holds.

Now we show the second equality holds.
For  $i=1, \ldots, n-1$,  \begin{align*}
\mathrm{Tr}(h_{n}J_{n+1}^ks_i)=\mathrm{Tr}(J_{n+1}^k)\mathrm{Tr}(h_{n}s_i)
=\mathrm{Tr}(J_{n+1}^k)\mathrm{Tr}(s_ih_{n})
=\mathrm{Tr}(s_ih_{n}J_{n+1}^k),
\end{align*}
where the middle equality uses the induction hypothesis $\mathrm{Tr}(h_{n}s_i)=\mathrm{Tr}(s_ih_{n})$.

Now Lemma~\ref{Lemm:J-s_n} and Eq.~\eqref{Equ:tr-def-JM} show
\begin{align*}
\mathrm{Tr}(h_{n}J_{n+1}^ks_n)&=z\,\mathrm{Tr}(h_{n}J_n^{k})+
\sum_{i=0}^{k-1}\mathrm{Tr}(h_nJ_{n}^iJ_{n+1}^{k-1-i}).
\end{align*}

On the other hand,  if $h_{n}\in H_{n-1}(\boldsymbol{u})$ then $s_nh_n=h_ns_n$ and Lemma~\ref{Lemm:J-s_n} shows \begin{eqnarray*}
\mathrm{Tr}(s_nh_{n}J_{n+1}^k)&=&z\,\mathrm{Tr}(h_{n}J_n^{k})+
\sum_{i=0}^{k-1}\mathrm{tr}(h_nJ_{n}^iJ_{n+1}^{k-1-i}).
\end{eqnarray*}

If $h_{n}=\alpha s_{n-1}\beta\in H_{n}(\boldsymbol{u})$ where $\alpha,\beta\in H_{n-1}(\boldsymbol{u})$, then
\begin{align*}\mathrm{Tr}(s_n\alpha s_{n-1}\beta J_{n+1}^k)
&=\mathrm{Tr}(\alpha (s_n J_{n+1}^k)s_{n-1}\beta)\\
&=z\,\mathrm{Tr}(\alpha J_n^{k}s_{n-1}\beta)+
\sum_{i=0}^{k-1}\mathrm{Tr}(\alpha J_{n}^iJ_{n+1}^{k-1-i}s_{n-1}\beta)\\
&=z\,\mathrm{Tr}(h_{n}J_n^{k})+
\sum_{i=0}^{k-1}\mathrm{Tr}(h_nJ_{n}^iJ_{n+1}^{k-1-i}),
\end{align*}
where the last equality follows by an induction argument.

If $h_{n}=\alpha J_{n}^{\ell}\in H_{n}(\boldsymbol{u})$ where $\alpha\in H_{n-1}(\boldsymbol{u})$ and $1\leq \ell\leq m-1$, then
\begin{align*}
\mathrm{Tr}(s_n\alpha J_n^{\ell} J_{n+1}^k)&=\mathrm{Tr}(\alpha (s_n J_{n+1}^k)J_n^{\ell})\\
&=z\,\mathrm{Tr}(\alpha J_n^{k+\ell})+\sum_{i=0}^{k-1}\mathrm{Tr}(\alpha J_{n}^iJ_{n+1}^{k-1-i}J_n^{\ell})\\
&=z\,\mathrm{Tr}(h_{n}J_n^{k})+
\sum_{i=0}^{k-1}\mathrm{Tr}(h_nJ_{n}^iJ_{n+1}^{k-1-i}).
\end{align*}
Thus $\mathrm{Tr}(h_{n}J_{n+1}^ks_n)=\mathrm{Tr}(s_nh_{n}J_{n+1}^k)$ holds in all cases.  It completes the proof.
\end{proof}

We are ready to show $\mathrm{Tr}$ is a trace function on $H_{\infty}(\boldsymbol{u})$.

\begin{theorem}\label{Them:ab=ba}For all $a,b\in H_{\infty}(\boldsymbol{u})$,
\begin{align*}\mathrm{Tr}(ab)=\mathrm{Tr}(ba).
\end{align*}
  \end{theorem}

\begin{proof}We verify $\mathrm{Tr}(ax)=\mathrm{Tr}(xa)$ inductively. Assume it holds for all $a, b\in H_{n}(\boldsymbol{u})$. For $a, b\in H_{n+1}(\boldsymbol{u})$, it suffices to check when $b$ is a generator $s_i$ ($1\leq i\leq n$) or $t$, that is, \begin{align*}(*)\qquad\quad\mathrm{Tr}(at)=\mathrm{Tr}(ta) \text{ and } \mathrm{Tr}(as_i)=\mathrm{Tr}(s_ia)\text{ for }1\leq i\leq n.
\end{align*}
Thanks to Theorem~\ref{Them:Inductive-basis}, we only need to check that Eq.~$(*)$ holds when $a$ is an element of the four forms (I)---(IV):
 If $a$ is an element of form (I), i.e., $a=h_n\in H_n(\boldsymbol{u})$, then the induction hypothesis gives
\begin{equation*}
\mathrm{Tr}(h_nt)=\mathrm{Tr}(th_n)\text{ and }\mathrm{Tr}(h_ns_i)=\mathrm{Tr}(s_ih_n)\text{ for }i=1, \ldots, n-1.
\end{equation*}
 Further Proposition~\ref{Prop:Tr-M}(M3) shows
\begin{equation*}
  \mathrm{Tr}(as_n)=z\,\mathrm{Tr}(a)=\mathrm{Tr}(s_na).
\end{equation*}
If $a$ is an element of forms (II)---(IV), the assertion follows directly form Lemmas~\ref{Lemm:JM-xs_nys_n}, \ref{Lemm:JM-Typer-III}, and  \ref{Lemm:JM-Type-IV}.
\end{proof}

Now we can prove Theorem~\ref{Them:Markov-BK} in Introduction.
\begin{proof}[Proof of Theorem~\ref{Them:Markov-BK}] (M1), (M3)--(M5) follow directly by Eq.~\eqref{Equ:tr-def-JM} and (M2) follows by applying Theorem~\ref{Them:ab=ba}. Note that having proved the existence, the uniqueness of $\mathrm{Tr}$ follows immediately. Indeed Theorem~\ref{Them:Inductive-basis} shows that $\mathrm{Tr}(x)$ can be clearly computed inductively using rules (M1)--(M4) for any $x\in H_{n+1}(\boldsymbol{u})$.
\end{proof}

\section{Specializations}
This section is devoted to investigating the specializations of (non-)normalized Markov traces. In particular, we obtain the Brou\'{e}--Malle--Michel symmetrizing trace on $H_n(\boldsymbol{u})$ and show that the Brundan--Kleshchev trace on $H_n(\boldsymbol{u})$ is a specialization of the non-normalized Markov trace.

We begin with the following fact.
\begin{lemma}\label{Lemm:t-trace}Let $\mathrm{tr}$ be the normalized Markov trace on $H_n(\boldsymbol{u})$ with parameters $z=0$, $y_1, \ldots, y_{m-1}\in \mathbb{C}(\boldsymbol{u})$. For any $w\in \mathfrak{S}_n-\{1\}$ and $0\leq a_1, \ldots, a_n\leq m-1$, we have
 \begin{align*}
\mathrm{tr}(t_1^{a_1}t_2^{a_2}\cdots t_n^{a_n}w)=0.
 \end{align*}
  \end{lemma}
 \begin{proof}Any reduced word $w\in\mathfrak{S}_n-\{1\}$ can be expressed in Jones' normal form:
 \begin{equation*}
  w=(s_{i_r}s_{i_r-1}\cdots s_{k_r})\cdots(s_{i_1}s_{i_1-1}\cdots s_{k_1}),
 \end{equation*}
   where $1\leq i_1<\cdots<i_r\leq n-1$, $1\leq k_1<\cdots<k_r\leq n-1$ and $i_j\geq k_j$ for all $j$.

  If $i_r<{n-1}$ then Theorem~\ref{Them:Markov}(m4) shows \begin{align*}\mathrm{tr}(t_1^{a_1}t_2^{a_2}\cdots t_n^{a_n}w)=y_{a_n}\mathrm{tr}(t_1^{a_1}t_2^{a_2}\cdots t_{n-1}^{a_{n-1}}w).\end{align*}
 So we may assume that $i_r=n-2$ and let $w'=s_{i_r}w$. Then $w'\in\mathfrak{S}_{n-2}$ and \begin{align*}
 \mathrm{tr}(t_1^{a_1}t_2^{a_2}\cdots t_{n-1}^{a_{n-1}}w)&=\mathrm{tr}\left(t_1^{a_1}t_2^{a_2}\cdots t_{n-2}^{a_{n-2}}s_{n-2}t_{n-2}^{a_{n-1}}w'\right)\\
 &=z\,\mathrm{tr}\left(t_1^{a_1}t_2^{a_2}\cdots t_{n-2}^{a_{n-2}+a_{n-1}}w'\right)\\
 &=0.
 \end{align*}
 If $i_r=n-1$ then $w'\in\mathfrak{S}_{n-1}$ and
 \begin{align*}
 \operatorname{tr}(t_1^{a_1}t_2^{a_2}\dotsm t_{n-1}^{a_{n-1}}w)
 &=\operatorname{tr}(t_1^{a_1}t_2^{a_2}\dotsm t_{n-1}^{a_{n-1}}s_{n-1}t_{n-1}^{a_n}w')\\
 &=z\operatorname{tr}(t_1^{a_1}t_2^{a_2}\dotsm t_{n-1}^{a_{n-1}+a_n}w')\\
 &=0.
 \end{align*}
In both cases, the trace vanishes under $z=0$. Thus, $\mathrm{tr}(t_1^{a_1}t_2^{a_2}\cdots t_n^{a_n}w)=0$ for all non-trivial $w$. This completes the proof.
 \end{proof}

The following result is an immediate consequence of Lemma~\ref{Lemm:t-trace}:

\begin{corollary}\label{Cor:t-trace}
Let $\mathrm{tr}_{\boldsymbol{0}}$ be the normalized Markov trace on $H_n(\boldsymbol{u})$ with parameters $z = y_1 = \cdots = y_{m-1} = 0$. Then $\mathrm{tr}_{\boldsymbol{0}}$ is a trace on $H_n(\boldsymbol{u})$ satisfying:
\begin{equation*}
  \mathrm{tr}_{\boldsymbol{0}}(t_1^{a_1} \cdots t_n^{a_n}w) =\begin{cases}
1, & \text{if } a_1 = \cdots = a_n = 0 \text{ and } w = 1; \\
0, & \text{otherwise}.
\end{cases}
\end{equation*}
\end{corollary}

\begin{proof}
Since $\mathrm{tr}_{\boldsymbol{0}}(1) = 1$ by definition, it remains to show:
\[
\mathrm{tr}_{\boldsymbol{0}}(t_1^{a_1} \cdots t_n^{a_n}w) = 0 \quad \text{for all } w \in \mathfrak{S}_n-\{1\} \text{ or } a_i \neq 0 \text{ (some } i\text{)}.
\]

\textit{Case 1:} When $w \in \mathfrak{S}_n-\{1\}$, Lemma~\ref{Lemm:t-trace} directly gives $\mathrm{tr}_{\boldsymbol{0}}(t_1^{a_1} \cdots t_n^{a_n}w) = 0$.

\textit{Case 2:} If $a_i \neq 0$ for some $i$, let $i$ be maximal with $a_i \neq 0$. By Theorem~\ref{Them:Markov} (m4), we have
\begin{equation*}
  \mathrm{tr}_{\boldsymbol{0}}(t_1^{a_1} \cdots t_n^{a_n}) = y_{a_i} \mathrm{tr}_{\boldsymbol{0}}(t_1^{a_1} \cdots t_{i-1}^{a_{i-1}}) = 0.
\end{equation*}
This completes the proof.
\end{proof}

\begin{remark}
This specialization $\mathrm{tr}_{\boldsymbol{0}}$, which is referred to as the Brou\'{e}--Malle--Michel symmetrizing trace on $H_n(\boldsymbol{u})$, closely resembles the canonical symmetric trace on cyclotomic Hecke algebras of type $G(m,1,n)$ defined by Bremke and Malle \cite{BM}. Their trace emerges as a specialization of the Markov trace for cyclotomic Hecke algebras, as detailed in \cite[Lemma~4.3]{GIM}. It is a natural question to determine the Schur elements of the degenerate cyclotomic Hecke algebras with respect to the Brou\'{e}--Malle--Michel symmetrizing trace.
\end{remark}

Let $\mathrm{tr}$ be the normalized Markov trace on $H_{\infty}(\boldsymbol{u})$ with parameters $z, y_1,\ldots, y_{m-1}$. It would be interesting to determine $\mathrm{tr}(J_1^{a_1}J_2^{a_2}\cdots J_i^{a_i}w)$ for all $1\leq a_i\leq m-1$,  for all $w\in _i$, and for all $i\geq 1$. For example, it is easy to verify that
\begin{equation*}
 \mathrm{tr}(J_1J_2\cdots J_i)=\prod_{j=1}^i\mathrm{tr}(J_i)=\prod_{j=1}^{i-1}(y_1+(j-1)z)
\end{equation*}
for $i=1,2\ldots$.
Furthermore, for positive integers $n,k$, and for all $h\in H_{n}(\boldsymbol{u})$, we have
\begin{align*}
    \mathrm{tr}(hJ_{n+1}\cdots J_{n+k})&=\mathrm{tr}(h)\prod_{i=1}^k(y_1+(n+i-1)z).
\end{align*}
Now assume that $z=0$, for positive integers $n,k$, and for all $h\in H_{n}(\boldsymbol{u})$, we have
\begin{equation*}
 \mathrm{tr}(hJ_{n+1}^2\cdots J_{n+k}^2)=\mathrm{tr}(h)\prod_{i=1}^{k}(y_2+n+i-1).
\end{equation*}

The following specialization of the non-normalized Markov traces will be helpful.
\begin{lemma}\label{Lemm:Specialization}Let $\mathrm{Tr}_0$ be the non-normalized Markov trace on $H_n(\boldsymbol{u})$ with parameters $z=0$ and $y_1, \ldots, y_{m-1}\in\mathbb{C}(\boldsymbol{u})$.  For $1\leq k\leq m-1$ and $1\leq i\leq n$,
  \begin{eqnarray*}\mathrm{Tr}_0(J_i^{k})&=&\mathrm{Tr}_0(J_{i-1}^{k})
+k\sum_{\ell=0}^{k-1}\mathrm{Tr}_0(J_{i}^{k-2-\ell}J_{i-1}^{\ell})-
\sum_{j=0}^{k-1}\sum_{\ell=0}^{j-1}\mathrm{Tr}_0(J_{i}^{k-2-\ell}J_{i-1}^{\ell}).
  \end{eqnarray*}
\end{lemma}
\begin{proof}Clearly it is trivial when $i=1$. For $i>1$, using Eq.~(\ref{Equ:JM-H}) and Lemma~\ref{Lemm:J-s_n}, we derive
\begin{align*}
J_{i}^{k}&=s_{i-1}J_{i-1}^{k}s_{i-1}+\sum_{j=0}^{k-1}J_{i}^{k-1-j}J_{i-1}^js_{i-1}\\
&=s_{i-1}J_{i-1}^{k}s_{i-1}+\sum_{j=0}^{k-1}J_{i}^{k-1-j}
\biggl(s_{i-1}J_i^j-\sum_{\ell=0}^{j-1}J_i^{j-1-\ell}J_{i-1}^{\ell}\biggr)\\
&=s_{i-1}J_{i-1}^{k}s_{i-1}+\sum_{j=0}^{k-1}J_{i}^{k-1-j}
s_{i-1}J_i^j-\sum_{j=0}^{k-1}\sum_{\ell=0}^{j-1}J_{i}^{k-2-\ell}J_{i-1}^{\ell}.
\end{align*}
Thus Theorem~\ref{Them:Markov-BK} shows
\begin{align*}
\mathrm{Tr}(J_{i}^{k})&=\mathrm{Tr}(J_{i-1}^{k})+k\,\mathrm{Tr}(J_{i}^{k-1}s_{i-1})-
\sum_{j=0}^{k-1}\sum_{\ell=0}^{j-1}\mathrm{Tr}(J_{i}^{k-2-\ell}J_{i-1}^{\ell}).
\end{align*}
Again using Lemma~\ref{Lemm:J-s_n} and Theorem~\ref{Them:Markov-BK}, we have
\begin{align*}
\mathrm{Tr}(J_{i}^{k-1}s_{i-1})&=z\,\mathrm{Tr}(J_{i-1}^{k})+
\sum_{\ell=0}^{k-1}\mathrm{Tr}(J_{i}^{k-2-\ell}J_{i-1}^{\ell}).
\end{align*}
Therefore, we yield
\begin{align*}
\mathrm{Tr}_0(J_{i}^{k})&=\mathrm{Tr}_0(J_{i-1}^{k})
+k\sum_{\ell=0}^{k-1}\mathrm{Tr}_0(J_{i}^{k-2-\ell}J_{i-1}^{\ell})-
\sum_{j=0}^{k-1}\sum_{\ell=0}^{j-1}\mathrm{Tr}_0(J_{i}^{k-2-\ell}J_{i-1}^{\ell}).
\end{align*}
It completes the proof.
\end{proof}

The following fact states the Brundan--Kleshchev trace $\tau_{\mathrm{BK}}$ is a specialization of the non-normalized Markov trace $\mathrm{Tr}$, which may viewed as a first step to the problem posed in \cite[Remark~5.7(ii)]{Z2}.
\begin{corollary}\label{Cor:BK-specialization}Let $\mathrm{Tr}_{\boldsymbol{0},1}$ be the non-normalized Markov trace on $H_n(\boldsymbol{u})$ with parameter $z=y_1=\cdots=y_{m-2}=0$ and $y_{m-1}=1$. Then $\mathrm{Tr}_{\boldsymbol{0},1}=\tau_{\mathrm{BK}}$.
\end{corollary}
\begin{proof}For $i=1, \ldots, n$, Lemma~\ref{Lemm:Specialization} shows   \begin{align*}\mathrm{tr}_{\boldsymbol{0},1}(J_{i}^{a_i})&=\left\{\begin{array}
  {ll}1,&\text{ if }a_i=m-1;\\
  0,&\text{ otherwise.}
\end{array}\right.\end{align*}
Therefore
\begin{align*}
 \mathrm{Tr}_{\boldsymbol{0},1}(J_1^{a_1}\cdots J_n^{a_n})=\left\{\begin{array}
  {ll}1,&\text{ if }a_1=\cdots=a_n=m-1;\\
  0,&\text{ otherwise.}
\end{array}\right.
\end{align*}
Note that any reduced word $w\in\mathfrak{S}_n-\{1\}$ can be expressed in Jones' normal form:
 \begin{equation*}
  w=(s_{i_r}s_{i_r-1}\cdots s_{i_r-k_r})\cdots(s_{i_1}s_{i_1-1}\cdots s_{i_1-k_1}),
 \end{equation*}
   where $1\leq i_1<\cdots<i_r\leq n-1$. Then Lemma~\ref{Lem:si-Jj} shows
  \begin{align*}
  J_1^{a_1}\cdots J_n^{a_n}w&=J_1^{a_1}\cdots J_{i_r}^{a_{i_r}}J_{i_r+1}^{a_{i_r+1}}s_{i_r}w'J_{i_r+2}^{a_{i_r+2}}\cdots J_n^{a_n},\end{align*}
  where $w'=s_{i_r}w$.
  Then Lemma~\ref{Lemm:J-s_n} shows
  \begin{align*}J_{i_r+1}^{a_{i_r+1}}s_{i_r}&= s_{i_r}J_{i_r}^{a_{i_r+1}}+\sum_{j=0}^{a_{i_r+1}-1}J_{i_r+1}^{a_{i_r+1}-1-j}J_{i_r}^j.
 \end{align*}
 Now Theorem~\ref{Them:Markov-BK} and $z=0$ show
 \begin{align*}
  \mathrm{Tr}_{\boldsymbol{0},1}(J_1^{a_1}\cdots J_n^{a_n}w)&=\sum_{j=0}^{a_{i_r+1}-1}\mathrm{Tr}_{\boldsymbol{0},1}(J_1^{a_1}\cdots J_{i_r}^{a_{i_r}+j}w'J_{i_r+1}^{a_{i_r+1}-1-j}J_{i_r+2}^{a_{i_r+2}}\cdots J_n^{a_n})\\
  &=\sum_{j=0}^{a_{i_r+1}-1}\mathrm{Tr}_{\boldsymbol{0},1}(J_1^{a_1}\cdots J_{i_r}^{a_{i_r}+j}w')\mathrm{Tr}_{\boldsymbol{0},1}(J_{i_r+1}^{a_{i_r+1}-1-j}J_{i_r+2}^{a_{i_r+2}}\cdots J_n^{a_n}).\end{align*}
 Thus apply the induction argument and trace properties iteratively show each term vanishes due to $z=0$ and Lemma~\ref{Lemm:Specialization}.
 It completes the proof.
\end{proof}

\subsection*{Acknowledgements} The authors are grateful to the anonymous referees  for their  numerous valuable comments, suggestions and corrections that have led to several improvements in both content and presentation.
\section*{Declaration}\noindent\textbf{Data Availability} Data sharing not applicable to this article as no datasets were generated or analysed during the current study.

\noindent\textbf{Conflict of Interest} There is no conflict of interest.


\end{document}